\documentclass[11pt]{article}
\usepackage{amsmath}
\usepackage{amssymb, yhmath}

\usepackage{enumerate}
\usepackage{listings}

\usepackage[pdftex]{graphicx}
\usepackage{epstopdf}

\usepackage{hyperref}
\usepackage[ruled]{algorithm2e}

\DeclareMathOperator{\sgn}{sgn}
\DeclareMathOperator{\Diag}{Diag}

\DeclareMathOperator{\erf}{erf}

\usepackage{color}
\definecolor{dkgreen}{rgb}{0,0.6,0}
\definecolor{gray}{rgb}{0.5,0.5,0.5}

\usepackage{authblk}

\date{}
\newtheorem{theorem}{Theorem}[section]

\newtheorem{lemma}[theorem]{Lemma}

\newtheorem{proposition}[theorem]{Proposition}

\newenvironment{proof}[1][Proof]{\begin{trivlist}\item[\hskip \labelsep {\bfseries #1.}]}{$\Box$\end{trivlist}}

\newcommand{\real}{\mathbb R}
\newcommand{\dd}{\mathrm{d}}
\newcommand{\twopartdef}[4]
{
    \left\{
        \begin{array}{ll}
            #1 & \mbox{, } #2 \\
            #3 & \mbox{, } #4
        \end{array}
    \right.
}
\numberwithin{equation}{section}

\title{Convolution based smooth approximations to the absolute value function 
with application to non-smooth regularization}

\author[1]{Sergey Voronin}
\author[ ]{G\"orkem \"Ozkaya}
\author[1]{Davis Yoshida}
\affil[1]{Department of Applied Mathematics, University of Colorado, Boulder, CO 80309, USA}

\date{\today}

\begin{document}

\maketitle

\begin{abstract}
We present new convolution based smooth approximations to the absolute value function 
and apply them to construct gradient based algorithms such as the nonlinear 
conjugate gradient scheme to obtain sparse, regularized solutions of linear 
systems $Ax = b$, a problem often tackled via 
iterative algorithms which attack the corresponding non-smooth minimization 
problem directly. In contrast, the approximations we propose allow us to replace 
the generalized non-smooth sparsity inducing functional by a smooth 
approximation of which we can readily compute gradients and Hessians. The resulting 
gradient based algorithms often yield a good estimate for the sought 
solution in few iterations and can either be used directly or to quickly warm 
start existing algorithms.
\end{abstract}

\section{Introduction}

Consider the linear system $Ax=b$, where
$A\in\mathbb{R}^{m\times n}$ and $b\in\mathbb{R}^m$. Often, in linear 
systems arising from physical inverse problems, we have 
more unknowns than data: $m \ll n$ \cite{tarantola82} and the right hand side of the system 
which we are given contains noise. In such a setting, 
it is common to introduce a constraint on the solution, both to account for the 
possible ill-conditioning of $A$ and noise in $b$ (regularization) and for the 
lack of data with respect to the number of unknown variables in the linear system.
A commonly used constraint is sparsity: to require the 
solution $x$ to have few nonzero elements compared to the dimension of $x$.
A common way of finding a sparse solution to 
the under-determined linear system
$Ax = b$ is to solve the classical Lasso problem \cite{DonohoSparseSolutions}. That is, 
to find $\bar{x} = \arg\min_x F_1(x)$ where
\[
   F_1(x) = \|Ax-b\|_2^2+2\tau\|x\|_1,
\]
i.e., the least squares problem 
with an $\ell_1$ regularizer, governed by the regularization parameter $\tau>0$ 
\cite{ingrid_thresholding1}.
For any $p\geq1$, the map 
$\|x\|_p := \left( \sum_{k=1}^n |x_k|^p \right)^{\frac{1}{p}}$
(for any $x\in\mathbb{R}^n$) is called the $\ell_p$-norm on 
$\mathbb{R}^n$. For $p=1$, the $||\cdot||_1$ norm is called an $\ell_1$ norm and is convex.
Besides $F_1(x)$, we also consider the more general $\ell_p$ functional (for 
$p > 0$) of which $\ell_1$ is a special case: 
\begin{equation}
\label{eq:lp_funct}
   F_p(x) = ||Ax - b||_2^2 
    + 2\tau \left( \sum_{k=1}^n |x_k|^p \right)^{\frac{1}{p}},
\end{equation}
As $p \to 0$, the right term of this functional approximates 
the count of nonzeros or the so called $\ell_0$ ``norm'' 
(which is not a proper norm): 
\begin{equation*}
   ||x||_0 
   = \lim_{p \to 0} \|x\|_p 
   = \lim_{p \to 0} 
        \left( \sum_{k=1}^n |x_k|^p \right)^{1/p},
\end{equation*}
which can be seen from Figure \ref{fig:fvals_to_diff_p}. 

\begin{figure*}[h!]
\centerline{
\includegraphics[scale=0.3]{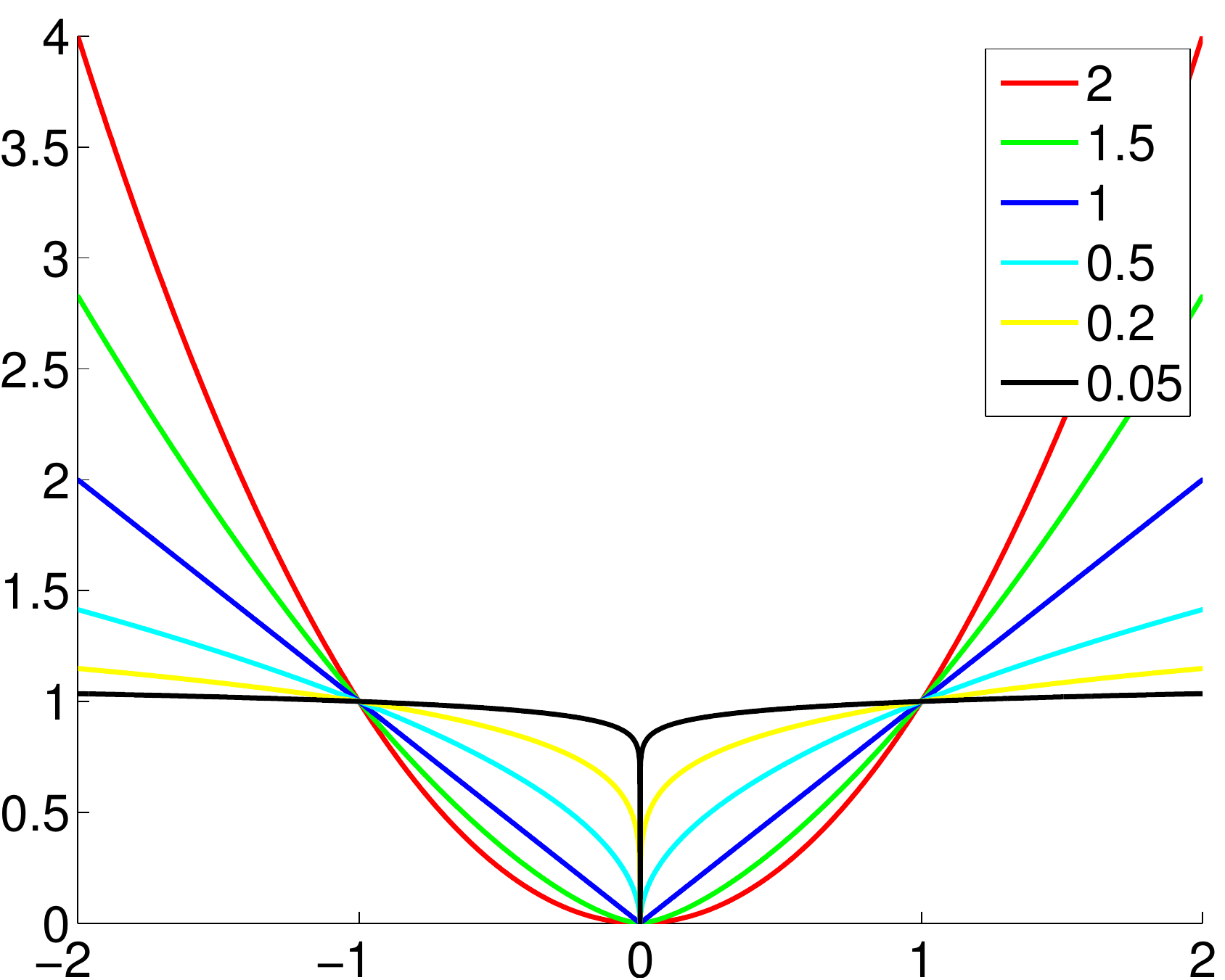}
}
\caption{$|x|^p$ plotted for different values of $p$, as $p \to 0$, the plot approaches an
indicator function.}
\label{fig:fvals_to_diff_p}
\end{figure*}

For $0<p<1$, \eqref{eq:lp_funct} is not convex. 
However, the minimization of non-smooth non-convex functions has been shown to 
produce good results in some compressive sensing applications 
\cite{Chartrand09fastalgorithms}. The non-smoothness of the functional $F_p(x)$, however, complicates 
its minimization from an algorithmic point of view. 
The non-smooth part of \eqref{eq:lp_funct}  
is due to the absolute value function $g(x_k) = |x_k|$. 
Because the gradient of $F_p(x)$ cannot be obtained, 
different minimization techniques such as sub-gradient methods are 
usually used \cite{ShorMinimization}.
For the convex $p=1$ case, various thresholding based 
methods have become popular. A particularly successful example 
is the soft thresholding 
based method FISTA \cite{MR2486527}. This algorithm is an accelerated version of a soft thresholded 
Landweber iteration \cite{MR0043348}:
\begin{equation}
\label{eq:ista}
x^{n+1} = \mathbb{S}_{\tau}(x^n + A^T b - A^T A x^n) 
\end{equation}
The soft thresholding function $\mathbb{S}_{\tau}:\mathbb{R}^n\to \mathbb{R}^n$ \cite{ingrid_thresholding1} is defined by
\begin{equation*}
   \left(\mathbb{S}_{\tau}(x)\right)_k = \sgn(x_k) \max{\{0, |x_k| - \tau\}}, \ \forall\, k=1,\ldots,n,\ \forall\, x\in\mathbb{R}^n.
\end{equation*}
The scheme \eqref{eq:ista} is known to converge from some initial guess, but slowly, to the 
$\ell_1$ minimizer \cite{ingrid_thresholding1}.
The thresholding in \eqref{eq:ista} is performed on 
$x^n - \nabla_x (\frac{1}{2} ||Ax^n - b||_2^2) = x^n - A^T (A x^n - b)$, which is a very 
simple gradient based scheme with a constant line search \cite{EnglRegularization}. 
Naturally, more advanced gradient schemes may be able to provide better numerical performance; however, they are possible to utilize only if we are able to 
compute the gradient of the functional we want to minimize.

In this article, we present new smooth approximations to the non-smooth 
absolute value function $g(t) = |t|$, computed via convolution with a Gaussian function.
This allows us to replace the non-smooth objective function 
$F_p(x)$ by a smooth functional $H_{p,\sigma}(x)$, which is close to $F_p(x)$ in value 
(as the parameter $\sigma \to 0$). Since the approximating functional 
$H_{p,\sigma}(x)$ is smooth, 
we can compute its gradient vector $\nabla_x H_{p,\sigma}(x)$ and Hessian matrix 
$\nabla^2_x H_{p,\sigma}(x)$. We are then able to use gradient based 
algorithms such as conjugate gradients to approximately minimize $F_p(x)$ 
by working with the approximate functional and gradient pair. The resulting 
gradient based methods we show are simple to implement and in many instances yield good 
numerical performance in few iterations. 

We remark that this article is not the first in attempting to use  
smooth approximations for sparsity constrained problems. A smooth $\ell_0$ norm 
approach has been  proposed in \cite{SmoothL0}. In this article, we propose a more 
general method which can be used for $F_p(x)$, including the popular $\ell_1$ case. 
The absolute value function which appears in non-smooth regularization is just one application 
of the convolution based smoothing approach we introduce here, which can likely be extended to 
different non-smooth functions.

\vspace{2.mm}

\section{Smooth approximation of absolute value function}

\subsection{Some existing smoothing techniques}
One simple smooth approximation to the absolute value function is given by 
$s_{\sigma}(t) = \sqrt{t^2 + \sigma^2}$ with $\sigma > 0 \in \mathbb{R}$.
\begin{lemma}
\label{lem:s_fcn_def}
The approximation $s_{\sigma}(t)$ to $|t|$ satisfies:
\begin{equation}
\label{eq:s_fcn_derivative}
\frac{\dd}{\dd t} s_{\sigma}(t) = t \left(t^2 + \sigma^2 \right)^{-\frac{1}{2}}
\end{equation}
and
\begin{equation}
\label{eq:s_fcn_approx_error}
\left||t| - s_{\sigma}(t)\right| \leq \sigma 
\end{equation}
\end{lemma}
\begin{proof}
To establish \eqref{eq:s_fcn_approx_error}, consider the inequality:
\begin{equation*}
\frac{t^2}{\sqrt{t^2 + \sigma^2}} \leq |t| \leq \sqrt{t^2 + \sigma^2}
\end{equation*}
It follows that:
\begin{equation*}
\left||t| - s_{\sigma}(t)\right| \leq \sqrt{t^2 + \sigma^2} - \frac{t^2}{\sqrt{t^2 + \sigma^2}} = \frac{\sigma^2}{\sqrt{t^2 + \sigma^2}} \leq \sigma
\end{equation*}
\end{proof}

Another well known smoothing technique for the absolute value is the so called Huber function 
\cite{BeckSmoothing}.
\begin{lemma}
\label{lem:huber_fcn_def}
The Huber function defined as:
\begin{equation}
\label{eq:huber_fcn}
p_{\sigma}(t) = \twopartdef { \frac{t^2}{2 \sigma} } {|t| \leq \sigma} {|t| - \frac{\sigma}{2}} {|t| \geq \sigma} 
\end{equation}
corresponds to the minimum value of the function 
\begin{equation}
\label{eq:ell1_reg_1d}
\min_{x \in \mathbb{R}} \left\{ \frac{1}{2 \sigma}(x - t)^2 + |x| \right\}
\end{equation}
It follows that: 
\begin{equation}
\label{eq:huber_fcn_derivative}
\frac{\dd}{\dd t} p_{\sigma}(t) = \twopartdef{ \frac{t}{\sigma} } {|t| \leq \sigma} {sgn(t)} {|t| \geq \sigma}
\end{equation}
and 
\begin{equation}
\label{eq:huber_fcn_approx_error}
\left| |t| - p_{\sigma}(t) \right| \leq \frac{\sigma}{2}
\end{equation}
\end{lemma}
\begin{proof}
The derivation follows by means of the soft thresholding operator \eqref{eq:thresholding}, 
which is known to satisfy the relation 
$\mathbb{S}_{\sigma}(t) = \arg\min_x \left\{ (x-t)^2 + 2\sigma |x| \right\}$ 
\cite{ingrid_thresholding1}. When $|t| \leq \sigma$, $\mathbb{S}_{\sigma}(t) = 0$. Plugging 
this into \eqref{eq:ell1_reg_1d}, we obtain: 
\begin{equation*}
\min \left\{ \frac{1}{2 \sigma}(x - t)^2 + |x| \right\} = \frac{1}{2 \sigma} (0 - t)^2 + |0| = 
\frac{t^2}{2 \sigma}
\end{equation*}
When $|t| > \sigma$, $S_{\sigma}(t) = t - \sigma$ (when $t > \sigma$) or $t + \sigma$ (when $t < -\sigma$). 
Taking the case $t > \sigma$, we have $|t| = t$, $|t - \sigma| = t - \sigma$ and:
\begin{equation*}
\min \left\{ \frac{1}{2 \sigma}(x - t)^2 + |x| \right\} = \frac{1}{2 \sigma} (t - \sigma - t)^2 + |t - \sigma| = \frac{1}{2 \sigma} \sigma^2 + (t - \sigma) = |t| - \frac{\sigma}{2}
\end{equation*}
Similarly, when $t < -\sigma$, we have $|t| = -t$, $|t + \sigma| = - t - \sigma$ and:
\begin{equation*}
\min \left\{ \frac{1}{2 \sigma}(x - t)^2 + |x| \right\} = \frac{1}{2 \sigma} (t + \sigma - t)^2 + |t + \sigma| = \frac{1}{2 \sigma} \sigma^2 -t - \sigma = |t| - \frac{\sigma}{2}
\end{equation*}
and so we obtain both parts of $\eqref{eq:huber_fcn}$. To show \eqref{eq:huber_fcn_approx_error}, 
consider both cases of $\eqref{eq:huber_fcn}$. When $|t| \geq \sigma$, 
$\left||t| - p_{\sigma}(t)\right| = \left||t| - |t| - \frac{\sigma}{2} \right| = \frac{\sigma}{2}$. When 
$|t| \leq \sigma$, we evaluate $\left| |t| - \frac{t^2}{2 \sigma}  \right|$. Let 
$u = |t|$, then for $u \leq \sigma$, $u - \frac{u^2}{2\sigma} > 0$ and 
$u - \frac{u^2}{2\sigma} = \frac{2 \sigma u - u^2}{2 \sigma}$. Let 
$r(u) = 2 \sigma u - u^2 \implies r^{\prime}(u) = 2 \sigma  - 2 u = 0$ 
and $r^{\prime\prime}(u) = -2 < 0$. 
Hence the max occurs at $u = \sigma$, which gives 
$\max\left( \frac{2 \sigma u - u^2}{2 \sigma} \right) = \frac{\sigma^2}{2 \sigma} = \frac{\sigma}{2}$ when $u = |t| \leq \sigma$.
\end{proof}

Lemmas \ref{lem:s_fcn_def} and \ref{lem:huber_fcn_def} imply that we can approximate the one norm of vector 
$x \in \mathbb{R}^n$ as $\|x\|_1 \approx \displaystyle\sum_{i=1}^n s_{\sigma}(x_i)$ or as 
$\|x\|_1 \approx \displaystyle\sum_{i=1}^n p_{\sigma}(x_i)$. 
From \eqref{eq:s_fcn_approx_error} and \eqref{eq:huber_fcn_approx_error}, the approximations 
satisfy:
\begin{equation*}
\displaystyle\sum_{k=1}^n s_{\sigma}(x_k) \leq \|x\|_1 \leq \displaystyle\sum_{k=1}^n s_{\sigma}(x_k) + \sigma n \quad \mbox{and} \quad
\displaystyle\sum_{k=1}^n p_{\sigma}(x_k) \leq \|x\|_1 \leq \displaystyle\sum_{k=1}^n p_{\sigma}(x_k) + \frac{\sigma n}{2}
\end{equation*} 
The smooth approximation for $\|x\|_1$ allows us to approximate the $\ell_1$ functional 
$F_1(x)$ as:
\begin{equation*}
S_{1,\sigma}(x) = ||Ax - b||_2^2 + 2 \tau \displaystyle\sum_{k=1}^n s_{\sigma}(x_k) \quad \mbox{and} \quad P_{1,\sigma}(x) = ||Ax - b||_2^2 + 2 \tau \displaystyle\sum_{k=1}^n p_{\sigma}(x_k)  
\end{equation*}
Note that from \eqref{eq:s_fcn_derivative} and \eqref{eq:huber_fcn_derivative}, the 
corresponding gradient vectors $\nabla S_{1,\sigma}(x)$ and $\nabla P_{1,\sigma}(x)$ are given by:
\begin{equation*}
\nabla S_{1,\sigma}(x) = 2 A^T (Ax - b) + 2 \tau \left\{ s^{\prime}_{\sigma}(x_k) \right\}_{k=1}^n \quad \mbox{and} \quad \nabla P_{1,\sigma}(x) = 2 A^T (Ax - b) + 2 \tau \left\{ p^{\prime}_{\sigma}(x_k) \right\}_{k=1}^n
\end{equation*}
with $s^{\prime}_{\sigma}(x_k) = x_k \left( x_k^2 + \sigma^2 \right)^{-\frac{1}{2}}$ and 
$p^{\prime}_{\sigma}(x_k) = \twopartdef{ \frac{x_k}{\sigma} } {|x_k| \leq \sigma} {sgn(x_k)} {|x_k| \geq \sigma}$. 
The advantage of working with the smooth functionals instead of 
$F_1(x)$ is that given the gradients we can use gradient based methods as we later describe. However, we cannot compute
the Hessian matrix of $P_{1,\sigma}(x)$ because $p_{\sigma}(x_k)$ is not twice 
differentiable, while $S_{1,\sigma}(x)$ is a less accurate approximation for $F_1(x)$. 
In the next section we introduce a new approximation to the absolute value based on convolution with 
a Gaussian kernel which addresses both of these concerns.

\subsection{Mollifiers and Approximation via Convolution}

In mathematical analysis, the concept of mollifiers is well known. 
Below, we state the definition and convergence result regarding mollifiers, 
as exhibited in \cite{DenkowskiNonlinearAnalysis}.
A smooth function $\psi_\sigma:\real\to\real$ 
is said to be a (non-negative) \textit{mollifier} 
if it has finite support,
is non-negative $(\psi \geq 0)$, 
and has area $\int_{\mathbb{R}} \psi(t) \dd t = 1$. 
For any mollifier $\psi$ and any $\sigma>0$, define the parametric function $\psi_\sigma:\mathbb{R}\to\mathbb{R}$ by: 
$\psi_{\sigma}(t) := \frac{1}{\sigma}\psi\left(\frac{t}{\sigma}\right)$, for all $t\in\real$.
Then $\{\psi_{\sigma}:\sigma>0\}$ is a family of mollifiers, 
whose support decreases 
as $\sigma \to 0$, but the volume under the graph always remains equal to one.
We then have the following important lemma for the approximation of 
functions, whose proof is given in \cite{DenkowskiNonlinearAnalysis}.
\begin{lemma}
For any continuous function $g\in L^1(\Theta)$ 
with compact support and $\Theta\subseteq\mathbb{R}$,
and any mollifier $\psi:\mathbb{R}\to\mathbb{R}$,
the convolution $\psi_{\sigma} * g$, which is the function defined
by:
\begin{equation*}
   (\psi_{\sigma} * g) (t)
   := \int_{\mathbb{R}} \psi_{\sigma} (t-s) g(s) \mathrm{d}s
   = \int_{\mathbb{R}} \psi_{\sigma} (s) g(t-s) \mathrm{d}s,
   \ \forall\, t\in\mathbb{R},
\end{equation*} 
converges uniformly to $g$ on $\Theta$, as $\sigma\to 0$.
\end{lemma}

Inspired by the above results, we will now use convolution with 
approximate mollifiers to approximate the absolute value function 
$g(t) = |t|$ (which is not in $L^1(\mathbb{R})$) with a smooth function.
We start with the Gaussian function $K(t) = \frac{1}{\sqrt{2 \pi}} \exp\left(-\frac{t^2}{2}\right)$ (for all $t\in\real$),
and introduce the $\sigma$-dependent family: 
\begin{equation}
\label{eq:gaussian_mollifier1}
   K_{\sigma}(t) := 
   \frac{1}{\sigma} K\left(\frac{t}{\sigma}\right) = 
   \frac{1}{\sqrt{2 \pi \sigma^2}} 
      \exp\left( -\frac{t^2}{2\sigma^2} \right), 
   \quad\forall\, t\in\mathbb{R}.
\end{equation} 
This function is not a mollifier because it does not have finite 
support. However, this function is coercive, that is, 
for any $\sigma>0$, $K_{\sigma}(t) \to 0$ as $|t|\to\infty$. 
In addition, we have that 
$\int_{-\infty}^{\infty} K_{\sigma}(t) \, \mathrm{d}t=1$ for all $\sigma>0$:
\begin{eqnarray*}
   \int_{-\infty}^{\infty} \! K_{\sigma}(t) \, \mathrm{d}t 
   &=& 
   \frac{1}{\sqrt{2 \pi \sigma^2}} 
     \int_{\mathbb{R}} \exp\left(-\frac{t^2}{2\sigma^2} \right) \, \mathrm{d}t 
    = 
   \frac{2}{\sqrt{2 \pi \sigma^2}} 
     \int_{0}^{\infty} \exp\left(-\frac{t^2}{2\sigma^2} \right) \, \mathrm{d}t 
\\ 
   &=& \frac{2}{\sqrt{2 \pi \sigma^2}} 
     \int_{0}^\infty \exp(-u^2) \sqrt{2} \sigma \, \mathrm{d}u
    = \frac{2}{\sqrt{2 \pi \sigma^2}} \sqrt{2} 
      \sigma \frac{\sqrt{\pi}}{2}
    = 1.
\end{eqnarray*}
Figure \ref{fig:Ksigma_plot} below presents a plot of the function 
$K_{\sigma}$ in relation to the particular choice $\sigma = 0.01$. 
We see that $K_{\sigma}(t) \geq 0$ and $K_{\sigma}(t)$ 
is very close to zero for $|t| > 4 \sigma$. 
In this sense, the function $K_\sigma$ is an approximate mollifier. 

\begin{figure*}[!ht]
\centerline{
\includegraphics[scale=0.35]{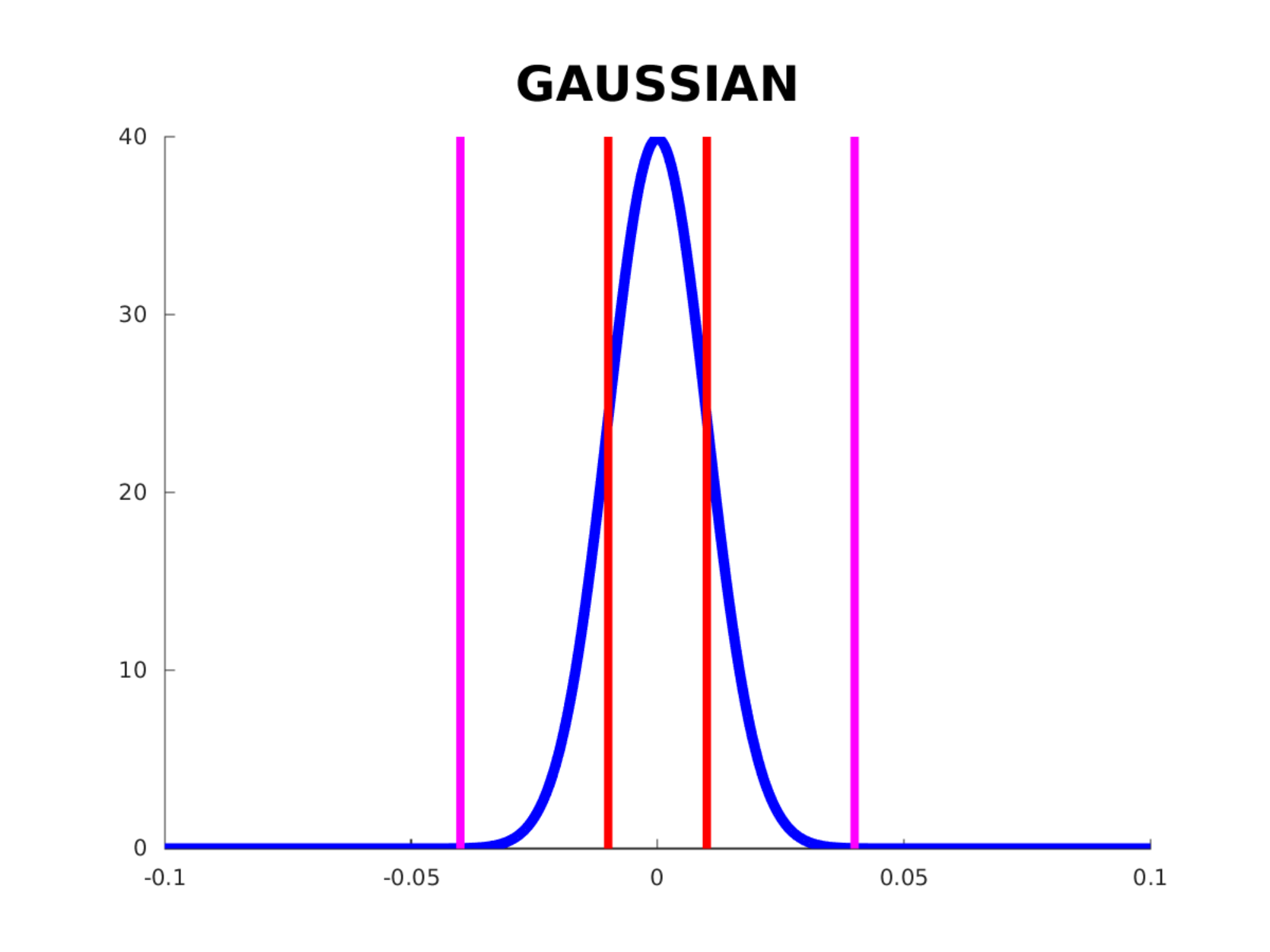}
}
\caption{$K_{\sigma}(t)$ and vertical lines at $(-\sigma, \sigma)$ and 
$(-4 \sigma, 4 \sigma)$ for $\sigma = 0.01$.}
\label{fig:Ksigma_plot}
\end{figure*}

Let us now compute the limit $\lim_{\sigma \to 0} K_{\sigma}(t)$. 
For $t = 0$, it is immediate that $\lim_{\sigma \to 0} K_{\sigma}(0) = \infty$. For $t\neq0$, we use 
l'H\^{o}spital's rule:
\begin{equation*}
   \lim_{\sigma \to 0} K_{\sigma}(t) = 
   \lim_{\sigma \to 0} \frac{1}{\sqrt{2 \pi \sigma^2}} 
     \exp\left( -\frac{t^2}{2\sigma^2} \right) = 
   \lim_{\gamma \to \infty} \frac{\gamma}{\sqrt{2 \pi} 
     \exp\left(\frac{\gamma^2t^2}{2}\right)} = 
   \frac{1}{\sqrt{2 \pi}} \lim_{\gamma \to \infty} 
     \frac{1}{\gamma\, t^2 \exp\left(\frac{\gamma^2t^2}{2}\right)} = 0,
\end{equation*}
with $\gamma = \frac{1}{\sigma}$. 
We see that $K_{\sigma}(t)$ behaves 
like a Dirac delta function $\delta_0(x)$ with unit integral over 
$\mathbb{R}$ and the same pointwise limit. 
Thus, for small $\sigma>0$, 
we expect that the absolute value function can be approximated 
by its convolution with $K_\sigma$, i.e., 
\begin{equation}
\label{eq:x_k_approx1}
   |t| \approx 
   \phi_\sigma(t),
   \quad\forall\, t\in\real,
\end{equation}
where the function $\phi_\sigma:\real\to\real$ is defined as the
convolution of $K_\sigma$ with the absolute value function:
\begin{equation}
\label{eq:phisigma}
   \phi_\sigma(t):=
   (K_{\sigma} * |\cdot|)(t) =
      \frac{1}{\sqrt{2 \pi \sigma^2}} 
      \int_{-\infty}^{\infty} |t - s|
      \exp\left( -\frac{s^2}{2\sigma^2} \right) 
      \mathrm{d}s,
   \quad\forall\, t\in\real.
\end{equation}
\noindent
We show in Proposition \ref{prop:conv} below, 
that the approximation in \eqref{eq:x_k_approx1} converges 
in the $L^1$ norm (as $\sigma \to 0$).
The advantage of using this approximation is that 
$\phi_\sigma$, unlike the absolute value function, 
is a smooth function. 

Before we state the convergence result in 
Proposition \ref{prop:conv}, 
we express the convolution integral and 
its derivative in terms of the well-known error function \cite{HandbookOfMathFunctions}.

\begin{lemma}
\label{lem:convolution}
For any $\sigma>0$, define $\phi_\sigma:\real\to\real$ 
as in \eqref{eq:phisigma}
Then we have that for all $t\in\real$:
\begin{align}
\label{eq:convolution}
   \phi_\sigma(t) 
   \,=&\ t \erf\left( \frac{t}{\sqrt{2} \sigma} \right) 
     + \sqrt{\frac{2}{\pi}} \sigma 
         \exp\left(-\frac{t^2}{2 \sigma^2}\right),
\\
\label{eq:phisigma_derivative}
   \frac{\dd}{\dd t}\,\phi_\sigma(t) 
   \,=&\ 
       \erf\left(\frac{t}{\sqrt{2} \sigma}\right),
\end{align} 
where the error function is defined as: 
\begin{equation*}
   \erf(t) = \frac{2}{\sqrt{\pi}} 
             \int_{0}^{t} \exp(-u^2) \dd u
   \quad\forall\, t\in\real.
\end{equation*}
\end{lemma}
\begin{proof}
Fix $\sigma>0$. 
Define $C_\sigma:\real_+\times\real$ by
\begin{eqnarray*}
   C_{\sigma}(T,t)
   &:=&
   \int_{-T}^{T} |t-s| K_\sigma(s) \dd s = 
   \frac{1}{\sqrt{2\pi\sigma^2}}
      \int_{-T}^{T} |t-s| \exp\left(-\frac{s^2}{2\sigma^2}\right)
      \dd s,\quad
   \forall\, t\in\mathbb{R}, \ T\geq0.
\end{eqnarray*}
We can remove the absolute value sign in the integration above by breaking up the integral into 
intervals from $-T$ to $t$ and from $t$ to $T$ where $|t - s|$ can be replaced by $(t - s)$ and 
$(s-t)$, respectively. Expanding the above we have that:
\begin{eqnarray*}
   \sqrt{2 \pi \sigma^2} C_{\sigma}(T,t) &=& 
   \int_{-T}^{T} |t-s|\exp\left(-\frac{s^2}{2\sigma^2}\right)\dd s 
\\
   &=& 
   \int_{-T}^{t} (t-s)\exp\left(-\frac{s^2}{2\sigma^2}\right) \dd s
      + \int_{t}^{T} (s-t)\exp\left(-\frac{s^2}{2\sigma^2}\right) \dd s 
\\
   &=&
   t\left(\int_{-T}^{t} \exp\left(-\frac{s^2}{2\sigma^2}\right) \dd s
      -\int_{t}^{T} \exp\left(-\frac{s^2}{2\sigma^2}\right) \dd s\right)
\\
   &&\qquad
   +\int_{-T}^{t} \exp\left(-\frac{s^2}{2\sigma^2}\right)(-s) \dd s
      +\int_{t}^{T} \exp\left(-\frac{s^2}{2\sigma^2}\right)s\, \dd s
\\
   &=&
   \sqrt2\sigma t \left(\int_{-T/\sqrt2 \sigma}^{t/\sqrt2 \sigma} \exp\left(-u^2\right) \dd u
      -\int_{t/\sqrt2 \sigma}^{T/\sqrt 2 \sigma} \exp\left(-u^2\right) \dd u\right)
\\
   &&\qquad
   +\sigma^2\left(\int_{-T}^{t} \exp\left(-\frac{s^2}{2\sigma^2}\right) \left(-\frac{s}{\sigma^2} \right) \dd s
      -\int_{t}^{T} \exp\left(-\frac{s^2}{2\sigma^2}\right) \left(-\frac{s}{\sigma^2} \right) \dd s\right).
\end{eqnarray*}
Next, making use of the definition of the error function, the fact that it's an odd function (i.e. 
$\erf(-t) = -\erf(t)$), and of the fundamental theorem of calculus, we have:
\begin{eqnarray*}
   \sqrt{2 \pi \sigma^2} C_{\sigma}(T,t) 
   &=& \sqrt{\frac{\pi}{2}}\sigma t\left(
		\erf\left(\frac{t}{\sqrt2\sigma}\right)
		-\erf\left(\frac{-T}{\sqrt2\sigma}\right)
		-\erf\left(\frac{T}{\sqrt2\sigma}\right)
		+\erf\left(\frac{t}{\sqrt2\sigma}\right)\right)
\\
   &&\qquad
	+\sigma^2\left(\int_{-T}^{t} \frac{\dd}{\dd s} \left(
	\exp\left(-\frac{s^2}{2\sigma^2}\right)
		\right) \dd s
	-\int_{t}^{T} \frac{\dd}{\dd s} 
      \left( \exp\left(-\frac{s^2}{2\sigma^2}\right) \right) \dd s \right)
\\
   &=&\sqrt{2\pi}\sigma t
		\erf\left(\frac{t}{\sqrt2\sigma}\right)
		+2\sigma^2\left(\exp\left(-\frac{t^2}{2\sigma^2}\right)
		-\exp\left(-\frac{T^2}{2\sigma^2}\right)\right).
\end{eqnarray*}
Since $\exp \left(-\frac{T^2}{2\sigma^2}\right) \to 0$ as $T \to \infty$, we have:
\begin{eqnarray*}
   \phi_\sigma(t)
   \,=\,
   (K_\sigma*|\cdot|)(t)
   &=&
   \lim_{T\to\infty} C_{\sigma}(T,t)= 
   \frac{1}{\sqrt{2 \pi} \sigma} 
      \sqrt{2\pi}\sigma t \erf\left(\frac{t}{\sqrt2\sigma}\right) +
   \frac{2}{\sqrt{2 \pi} \sigma} \sigma^2 
      \exp\left(-\frac{t^2}{2\sigma^2}\right) 
\\
   &=& 
   t\erf\left(\frac{t}{\sqrt2\sigma}\right)
      +\sqrt{\frac{2}{\pi}}\sigma
      \exp\left(-\frac{t^2}{2\sigma^2}\right).
\end{eqnarray*}
This proves \eqref{eq:convolution}.
To derive \eqref{eq:phisigma_derivative}, we use
\begin{equation}
\label{eq:erf_deriv}
   \frac{\dd}{\dd t} \erf\left(\frac{t}{\sqrt{2} \sigma}\right) = 
 \frac{2}{\sqrt{\pi}} 
            \frac{\dd}{\dd t} \left[ \int_{0}^{\left(\frac{t}{\sqrt{2} \sigma}\right)} \exp(-u^2) \dd u \right]
   = \frac{\sqrt{2}}{\sigma \sqrt{\pi}} 
      \exp\left(-\frac{t^2}{2 \sigma^2}\right),
\end{equation}
Plugging in, we get:
\begin{eqnarray*}
   \frac{\dd}{\dd t}\,\phi_{\sigma} (t)
   &=&
   \frac{\dd}{\dd t} \left( t \erf 
      \left( \frac{t}{\sqrt{2} \sigma} \right) 
      + \sqrt{\frac{2}{\pi}} \sigma \exp\left(-\frac{t^2}{2 \sigma^2}\right) \right) 
\\
   &=& 
   \erf\left(\frac{t}{\sqrt{2} \sigma}\right) 
       + t \frac{\sqrt{2}}{\sigma \sqrt{\pi}} 
       \exp\left(-\frac{t^2}{2 \sigma^2}\right) 
       - \sqrt{\frac{2}{\pi}} \sigma \frac{2 t} {2\sigma^2} 
       \exp\left(-\frac{t^2}{2\sigma^2}\right)
\\
   &=& 
   \erf\left(\frac{t}{\sqrt{2} \sigma}\right).
\end{eqnarray*}
so that \eqref{eq:phisigma_derivative} holds.
\end{proof}

\vspace{3.mm}
\noindent
Next, we review some basic properties of the error function 
$\erf(t) = \frac{2}{\sqrt{\pi}} \int_{0}^t \exp(-s^2) \dd s$ and the Gaussian integral 
\cite{HandbookOfMathFunctions}. 
It is well known that the Gaussian integral satisfies:
\begin{equation*}
    \int_{-\infty}^{\infty} \exp(-s^2) \dd s = 
    2 \int_{0}^{\infty} \exp(-s^2) \dd s = 
    \sqrt{\pi} ,
\end{equation*} 
and, in particular, $0 < \erf(t) < 1$ for all $t>0$. 
Using results from \cite{ChuNormalIntegral} on the Gaussian integral, 
we also have the following bounds for the error function:

\begin{lemma}
The error function $\erf(t) = \frac{2}{\sqrt{\pi}} \int_{0}^t \exp(-s^2) \dd s$ satisfies the bounds:
\begin{equation}
\label{eq:erf_bounds}
   \bigl(1 - \exp(-t^2)\bigr)^\frac{1}{2} \leq 
   \erf(t) \leq 
   \bigl(1 - \exp(-2 t^2)\bigr)^\frac{1}{2},
      \quad\forall\, t\geq0.
\end{equation}
\end{lemma}

\begin{proof}
In \cite{ChuNormalIntegral}, it is shown that for 
the function $v(t) := \frac{1}{\sqrt{2 \pi}} \int_{0}^t 
\exp\left(-\frac{s^2}{2}\right) \dd s$, 
the following bounds hold:
\begin{equation}
\label{eq:vx_bounds}
   \frac{1}{2} \left( 1 - \exp\left(-\frac{t^2}{2}\right) \right)^{\frac{1}{2}} \leq 
   v(t) \leq 
   \frac{1}{2} \left( 1 - \exp\left(-t^2\right) \right)^{\frac{1}{2}},
   \quad\forall\, t\geq0.
\end{equation}
Now we relate $v(t)$ to the error function. 
Using the substitution $u=\frac{s}{\sqrt{2}}$:
\begin{eqnarray*}
   v(t) = 
   \frac{1}{\sqrt{2 \pi}} 
      \int_{0}^{\frac{t}{\sqrt{2}}} \exp(-u^2) \sqrt{2} \dd u = 
   \frac{1}{2} \erf\left( \frac{t}{\sqrt{2}} \right) .
\end{eqnarray*}
From \eqref{eq:vx_bounds}, it follows that:
\begin{equation*}
   \left( 1 - \exp\left(-\frac{t^2}{2}\right) \right)^{\frac{1}{2}} \leq 
   \erf\left( \frac{t}{\sqrt{2}} \right) \leq 
   \left( 1 - \exp(-t^2) \right)^{\frac{1}{2}}.
\end{equation*}
With the substitution $s = \frac{t}{\sqrt{2}}$, we obtain \eqref{eq:erf_bounds}.
\end{proof}
Using the above properties of the error function and the Gaussian integral, 
we can prove our convergence result.

\begin{proposition}
\label{prop:conv}
Let $g(t):=|t|$ for all $t\in\real$, 
and let the function $\phi_\sigma:=K_\sigma*g$
be defined as in \eqref{eq:phisigma}, for all $\sigma>0$.
Then:
\begin{equation*}
\lim_{\sigma \to 0} \left\|\phi_{\sigma} - g\right\|_{L^1} = 0.
\end{equation*}
\end{proposition}

\begin{proof}
By definition of $\phi_\sigma$,
\begin{equation*}
   \left\|\phi_{\sigma} - g\right\|_{L^1}
   = 
   \int_{-\infty}^{\infty} \bigl| (K_{\sigma} * |\cdot|)(t) - |t|\,  \bigr| \mathrm{d} t = 2 \int_{0}^{\infty} \bigl| (K_{\sigma} * |\cdot|)(t) - t  \bigr| \mathrm{d} t ,
\end{equation*}
where the last equality follows from the fact that 
$\phi_\sigma-g$ is an even function since both $\phi_{\sigma}$ and 
$g$ are even functions. 
Plugging in \eqref{eq:x_k_approx1}, we have:
\begin{eqnarray*}
   ||\phi_{\sigma} - g||_{L^1} 
   &=& 
   2 \int_{0}^{\infty} \left| 
      t \left( \erf\left(\frac{t}{\sqrt{2} \sigma}\right) -1 \right) 
   + \sqrt{\frac{2}{\pi}} \sigma \exp\left(-\frac{t^2}{2 \sigma^2}\right) 
   \right| \mathrm{d} t 
\\
   &\leq& 
   2 \int_{0}^{\infty} 
   \left| t \left( \erf\left(\frac{t}{\sqrt{2} \sigma}\right) -1 \right) \right| 
   + \sqrt{\frac{2}{\pi}} \sigma \exp\left(-\frac{t^2}{2 \sigma^2}\right) 
      \mathrm{d} t 
\\
   &=&  
   2 \int_{0}^{\infty} 
      t \left( 1 - \erf\left(\frac{t}{\sqrt{2} \sigma} \right) \right) 
      + \sqrt{\frac{2}{\pi}} \sigma \exp\left(-\frac{t^2}{2 \sigma^2}\right) \mathrm{d} t ,
\end{eqnarray*}
where we have used the inequality $0 < \erf(t) < 1$ for $t > 0$. 
Next, we analyze both terms of the 
integral. First, using \eqref{eq:erf_bounds}, we have: 
\begin{eqnarray*}
   \left(1 - \exp(-t^2)\right)^\frac{1}{2} 
   \leq 
   \erf(t) 
\implies 
   1 - \erf(t) 
   \leq 
   1 - \left(1 - \exp(-t^2)\right)^\frac{1}{2} 
   \leq 
   \exp\left(-\frac{x^2}{2}\right),
\end{eqnarray*}
where the last inequality follows from the fact that $1 - \sqrt{1 - \alpha} \leq \sqrt{\alpha}$ 
for $\alpha \in (0,1)$. It follows that:
\begin{equation*}
   \int_{0}^{\infty} t
       \left( 1 - \erf\left(\frac{t}{\sqrt{2} \sigma} \right) \right) \mathrm{d}t 
   \leq 
   \int_{0}^{\infty} t \exp\left( -\frac{t^2}{4 \sigma^2} \right) \mathrm{d}t
   = 
   2 \sigma^2 \int_{0}^{\infty} \exp(-u) \mathrm{d}u
   = 
   2 \sigma^2,
\end{equation*} 
For the second term, 
\begin{equation*}
   \sqrt{\frac{2}{\pi}} \sigma \int_{0}^{\infty} 
      \exp\left(-\frac{t^2}{2 \sigma^2}\right) \mathrm{d} t 
   = 
   \sqrt{\frac{2}{\pi}} \sigma 
      \left(\sqrt{2} \sigma \int_{0}^{\infty} \exp(-s^2) \mathrm{d}s\right) 
   = 
   \frac{2}{\sqrt{\pi}} \sigma^2 \frac{\sqrt{\pi}}{2} 
   = 
   \sigma^2.
\end{equation*}
Thus:
\begin{equation*}
   \lim_{\sigma \to 0} \|\phi_{\sigma} - g\|_{L^1} 
   \leq 
   2 \lim_{\sigma \to 0} \left( 2 \sigma^2 + \sigma^2 \right)
   = 0.
\end{equation*}
Hence we have that $\lim_{\sigma \to 0} \|\phi_{\sigma} - g\|_{L^1} = 0$.
\end{proof}

Note that in the above proof, $g=|\cdot| \not \in L^1$ (since $g(t) \to \infty$ as 
$t \to \infty$), but the approximation in the $L^1$ norm still holds. It is likely 
that the convolution approximation converges to $g$ in the $L^1$ norm for a variety of non-smooth coercive functions $g$, 
not just for $g(t) = |t|$. 

%

Note from \eqref{eq:x_k_approx1} that 
while the approximation $\phi_{\sigma}(t) = K_\sigma*|t|$ is indeed smooth,
it is positive on $\real$ and in particular $(K_\sigma*|\cdot|)(0)=\sqrt{\frac{2}{\pi}} \sigma>0$, 
although $(K_\sigma*|\cdot|)(0)$ does go to zero as $\sigma \to 0$. 
To address this, we can use different approximations based on $\phi_{\sigma}(t)$ which are zero 
at zero. Many different alternatives are possible. We describe here two particular approximations. 
The first is formed by subtracting the value at $0$: 
\begin{equation}
\label{eq:x_k_approx2}
   \tilde{\phi}_{\sigma}(t) = \phi_\sigma(t)-\phi_\sigma(0)
   \,=\,
   t \erf \left( \frac{t}{\sqrt{2} \sigma} \right) 
   + \sqrt{\frac{2}{\pi}} \sigma \exp\left(\frac{-t^2}{2 \sigma^2}\right)
   - \sqrt{\frac{2}{\pi}} \sigma,
\end{equation}
An alternative is to use $\tilde{\phi^{(2)}}_{\sigma}(t) = \phi_\sigma(t) - \sqrt{\frac{2}{\pi}} \sigma \exp\left(-t^2\right)$ 
where the subtracted term decreases in magnitude as $t$ becomes larger and only has much effect for $t$ close to zero. 
We could also simply drop the second term of $\phi_{\sigma}(t)$ to get:
\begin{equation}
\label{eq:x_k_approx3}
   \hat{\phi}_{\sigma}(t) = \phi_\sigma(t) - \sqrt{\frac{2}{\pi}} \sigma \exp\left(\frac{-t^2}{2 \sigma^2}\right)
   \,=\,
   t \erf \left( \frac{t}{\sqrt{2} \sigma} \right) 
\end{equation}
which is zero when $t = 0$. 

\subsection{Comparsion of different approximations}
We now illustrate the different convolution based approximations along with 
the previously discussed $s_{\sigma}(t)$ and $p_{\sigma}(t)$. In 
Figure \ref{fig:smooth_approximations_for_abs_value}, we plot the absolute value function 
$g(t) = |t|$ and the different approximations 
$\phi_{\sigma}(t), \tilde{\phi}_{\sigma}(t), \hat{\phi}_{\sigma}(t), s_{\sigma}(t)$, 
and $p_{\sigma}(t)$ for $\sigma_1 = 3e^{-4}$ (larger value corresponding to a worser approximation) 
and $\sigma_2 = e^{-4}$ (smaller value corresponding to a better approximation) for a small range of 
values $t \in (-4.5 e^{-4}, 4.5 e^{-4})$ around $t = 0$. 
We may observe that $\phi_{\sigma}(t)$ smooths out the sharp corner 
of the absolute value, at the expense of being above zero at $t = 0$ for positive $\sigma$. The modified 
approximations $\tilde{\phi}_{\sigma}(t)$ and $\hat{\phi}_{\sigma}(t)$ are zero at zero. However, 
$\tilde{\phi}_{\sigma}(t)$ over-approximates $|t|$ for all $t > 0$ while 
$\hat{\phi}_{\sigma}(t)$ does not preserve convexity. The three $\phi$ approximations 
respectively capture the general characteristics possible to obtain with the 
described convolution approach. 
From the figure we may observe that $\phi_{\sigma}(t)$ and 
$\hat{\phi}_{\sigma}(t)$ remain closer to the absolute value curve than 
$s_{\sigma}(t)$ and $p_{\sigma}(t)$ as $|t|$ becomes larger. The best approximation 
appears to be $\hat{\phi}_{\sigma}(t)$, which is close to $|t|$ even for the larger 
value $\sigma_1$ and is twice differentiable.

\newpage

\begin{figure}[!ht]
\centerline{
\includegraphics[scale=0.25]{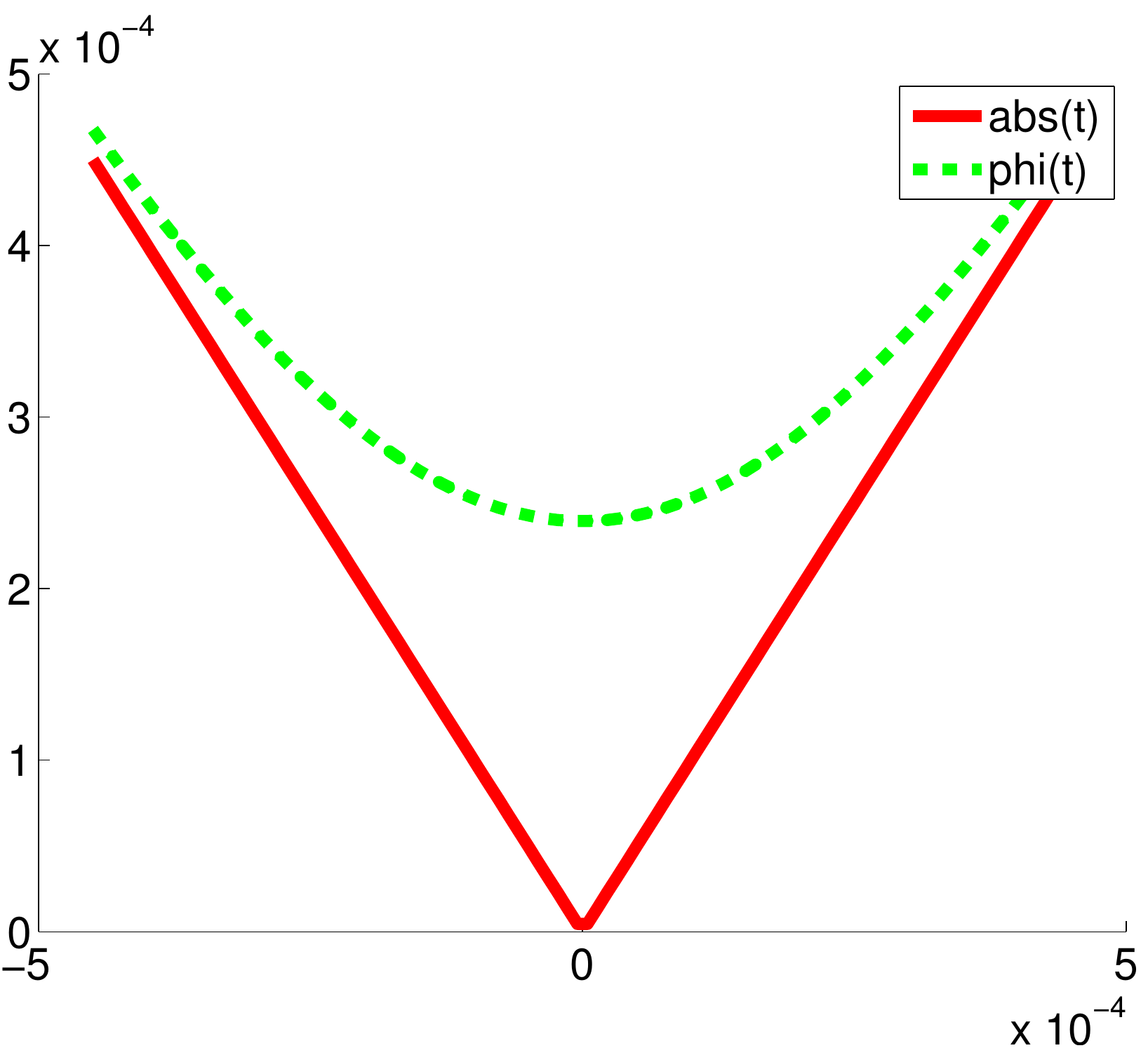} 
\includegraphics[scale=0.25]{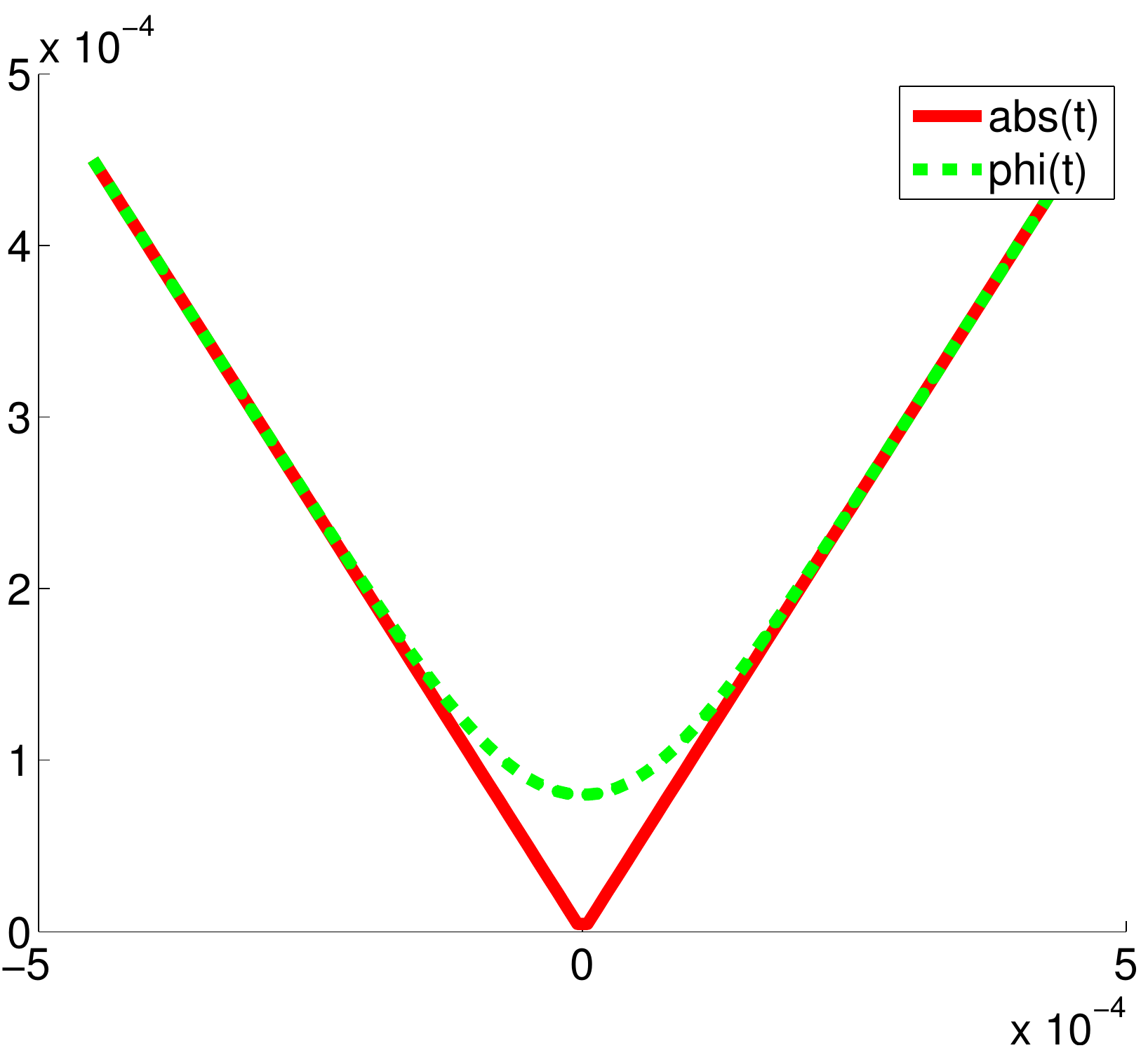} 
}
\centerline{
\includegraphics[scale=0.25]{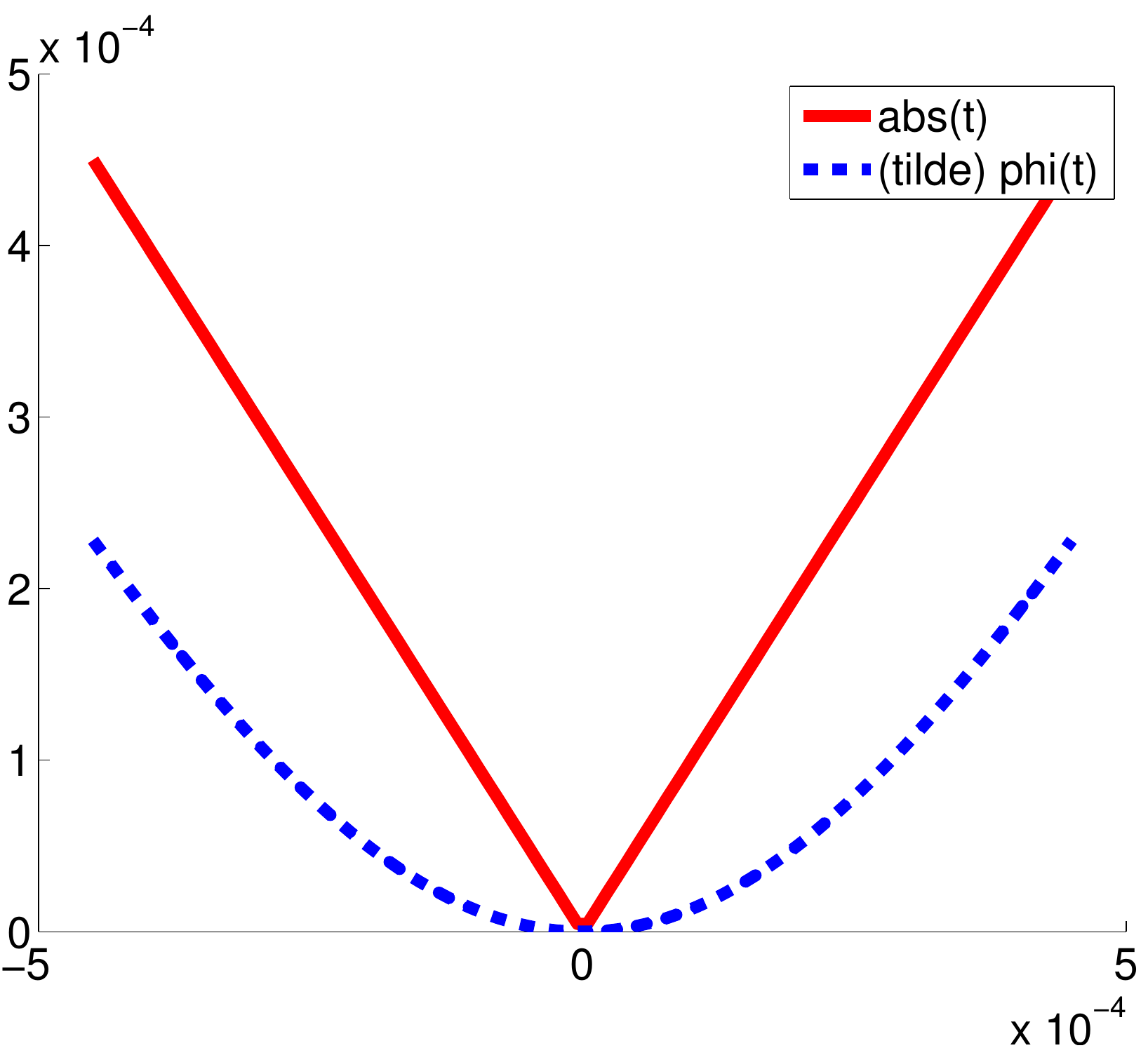} 
\includegraphics[scale=0.25]{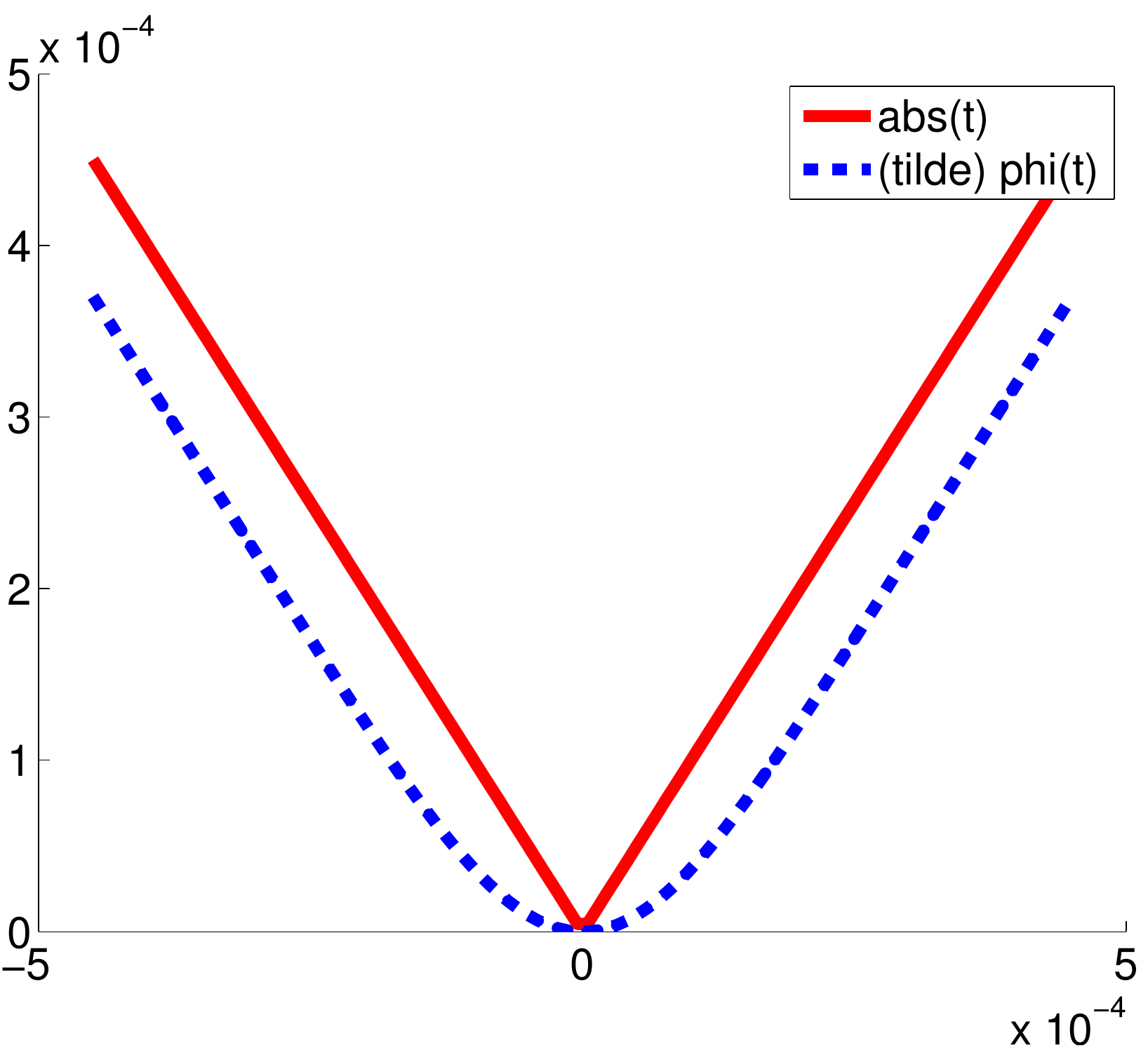} 
\includegraphics[scale=0.25]{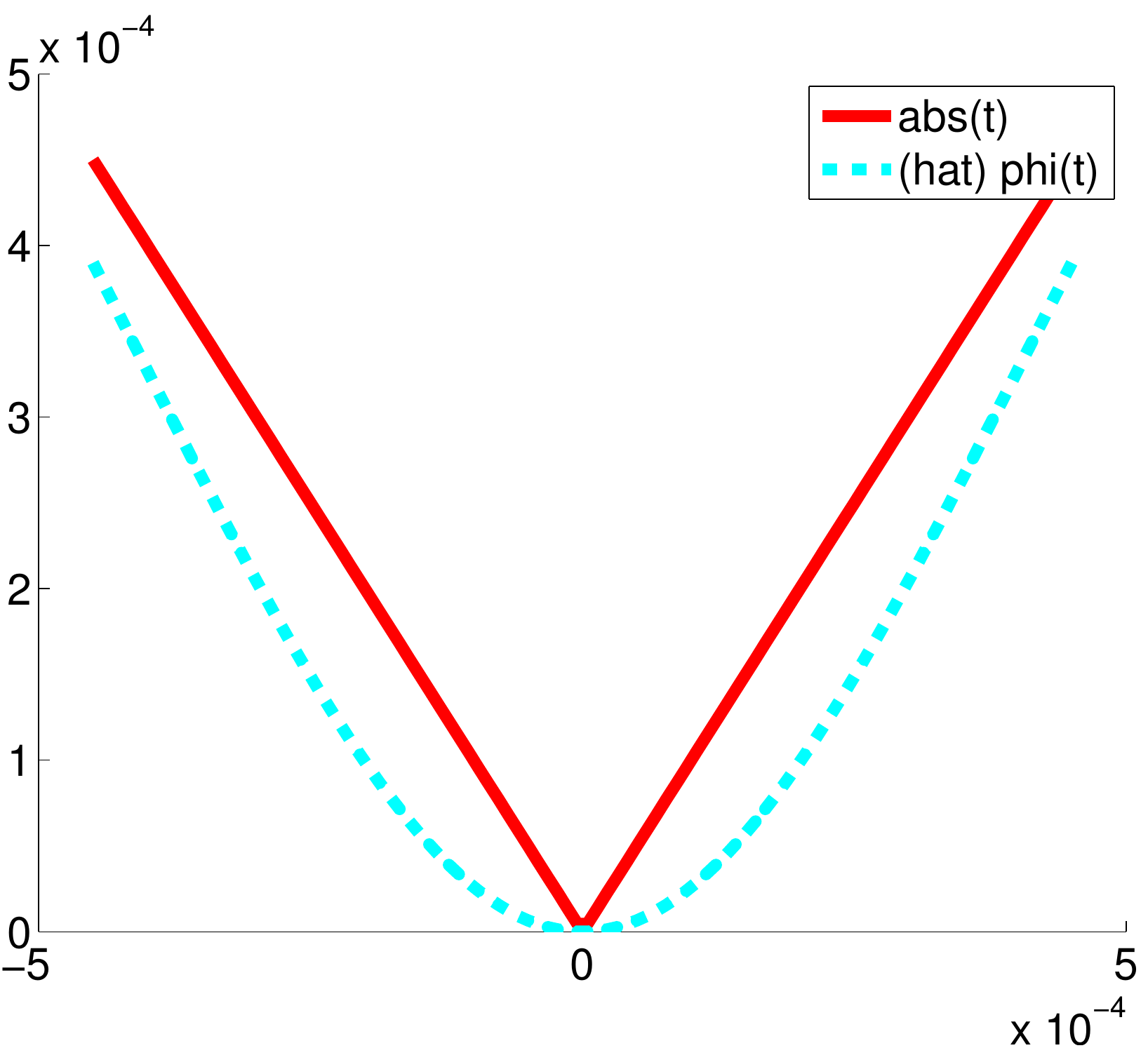} 
\includegraphics[scale=0.25]{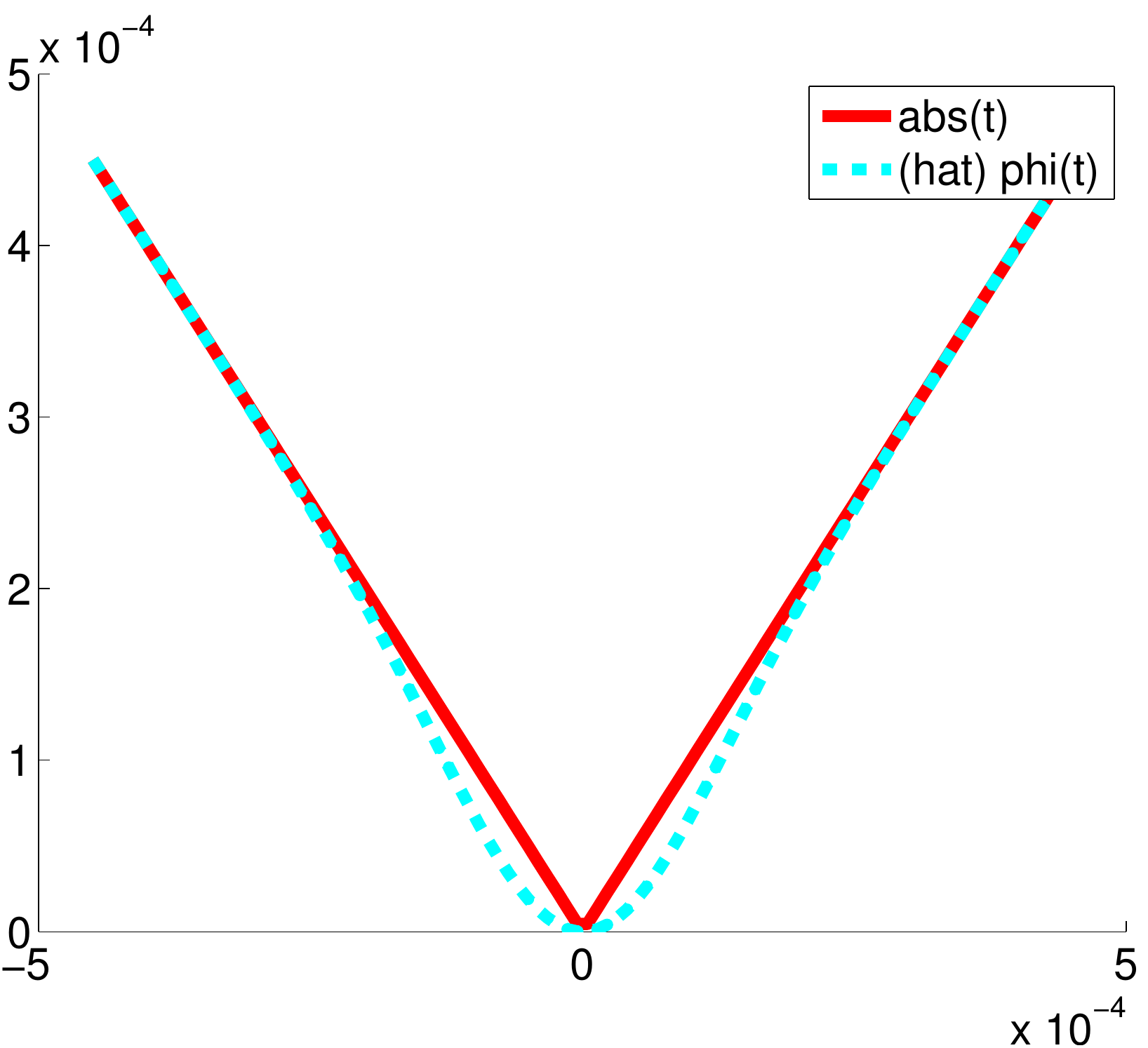} 
}
\centerline{
\includegraphics[scale=0.25]{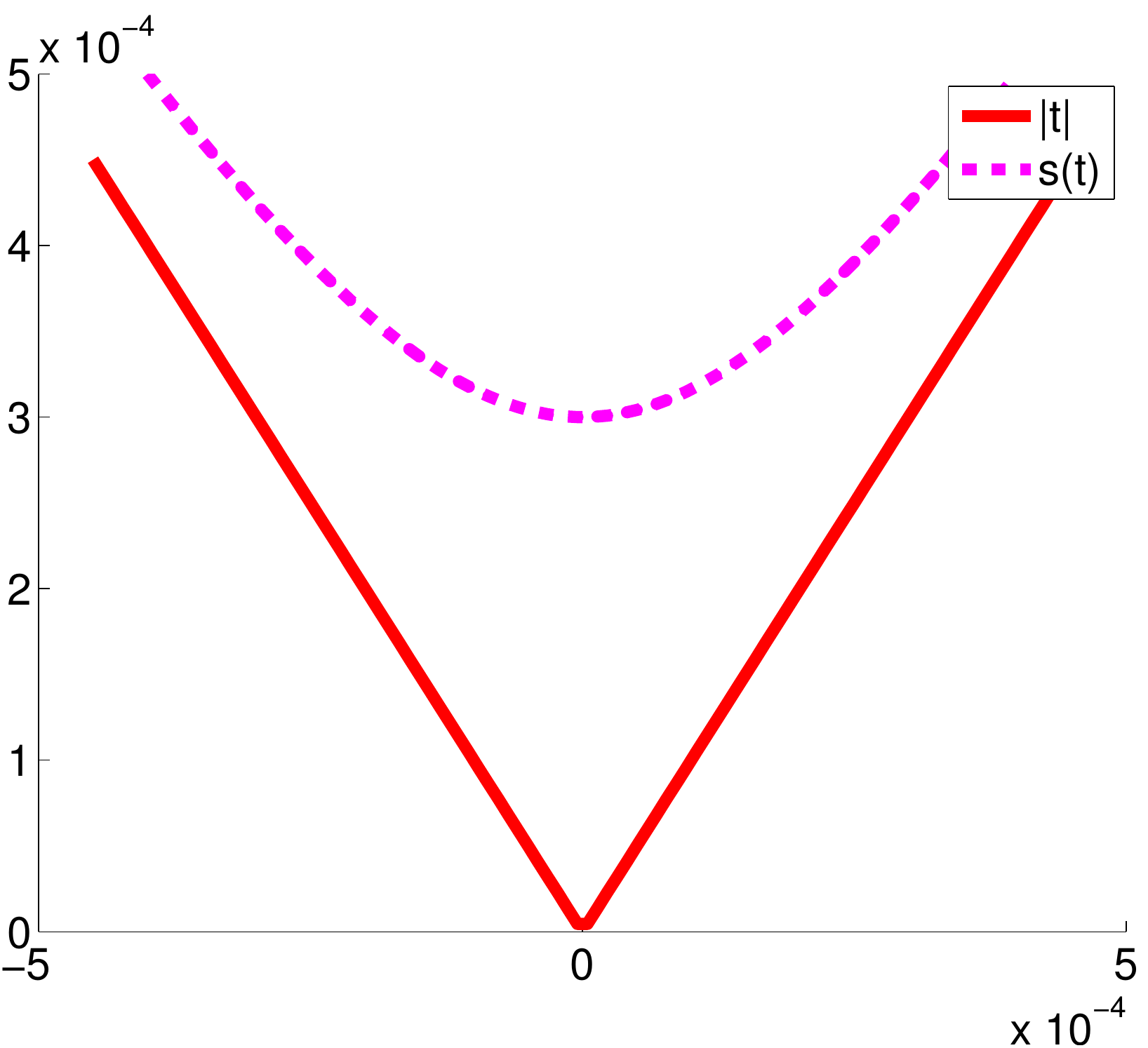} 
\includegraphics[scale=0.25]{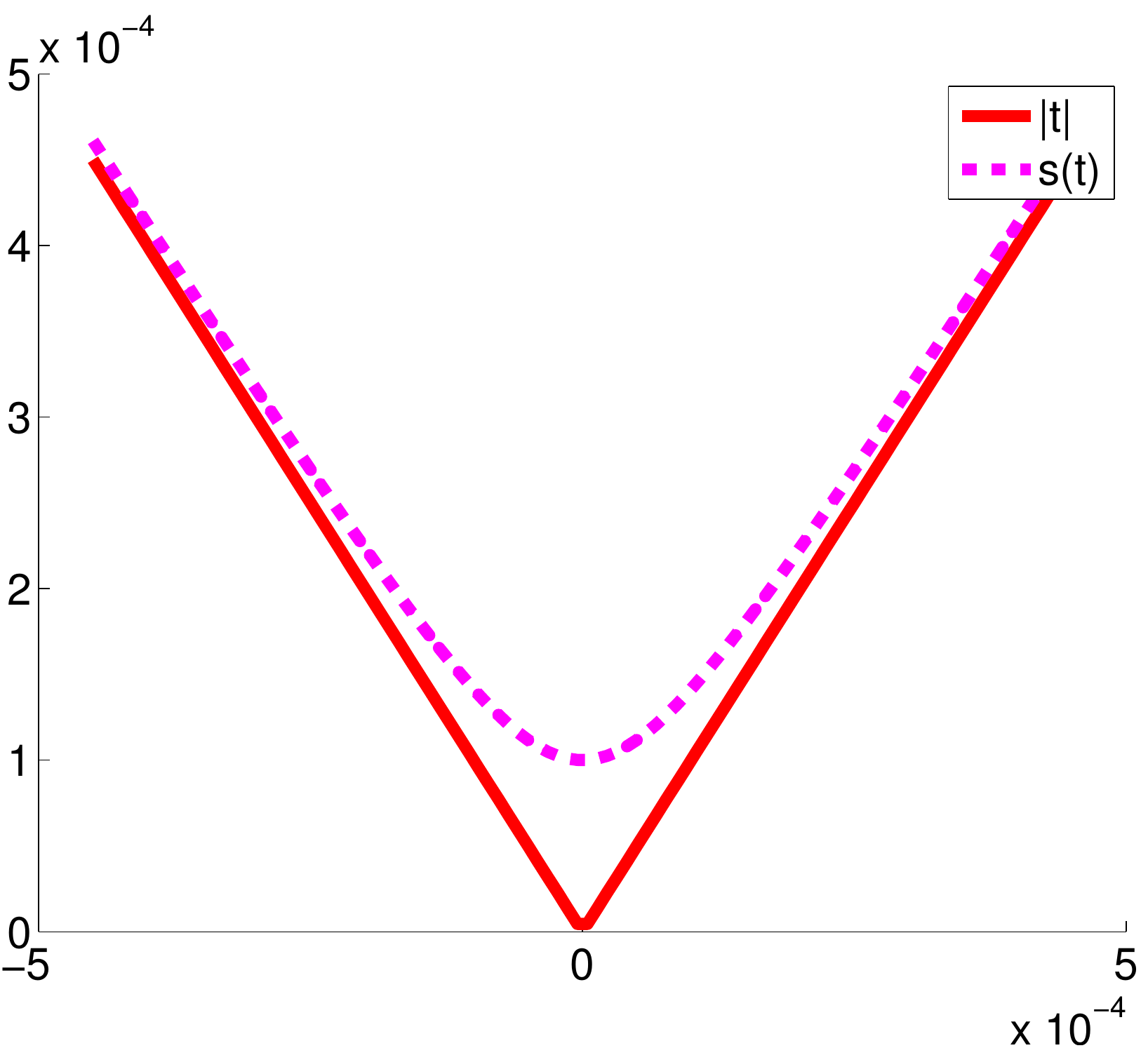} 
\includegraphics[scale=0.25]{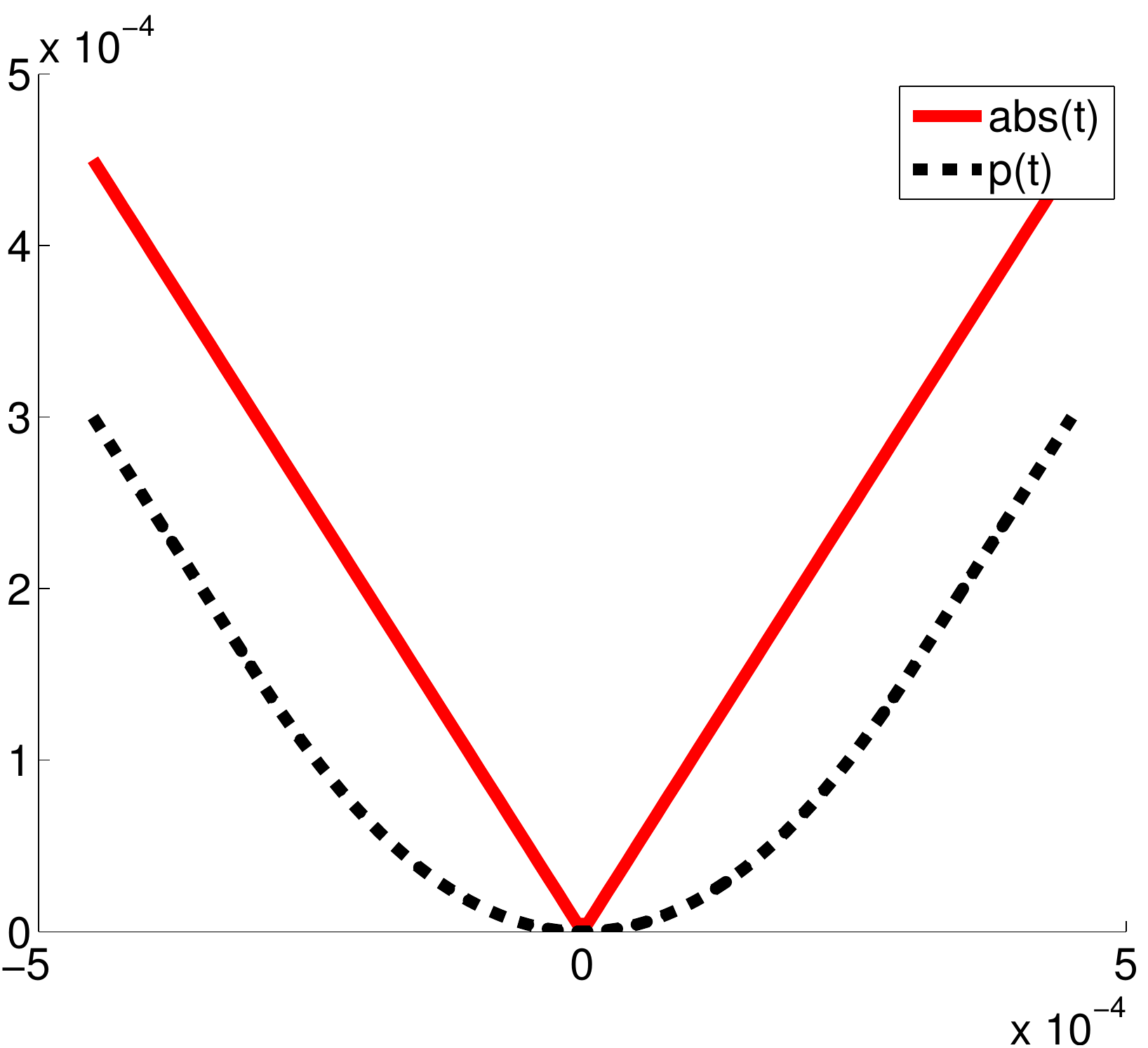} 
\includegraphics[scale=0.25]{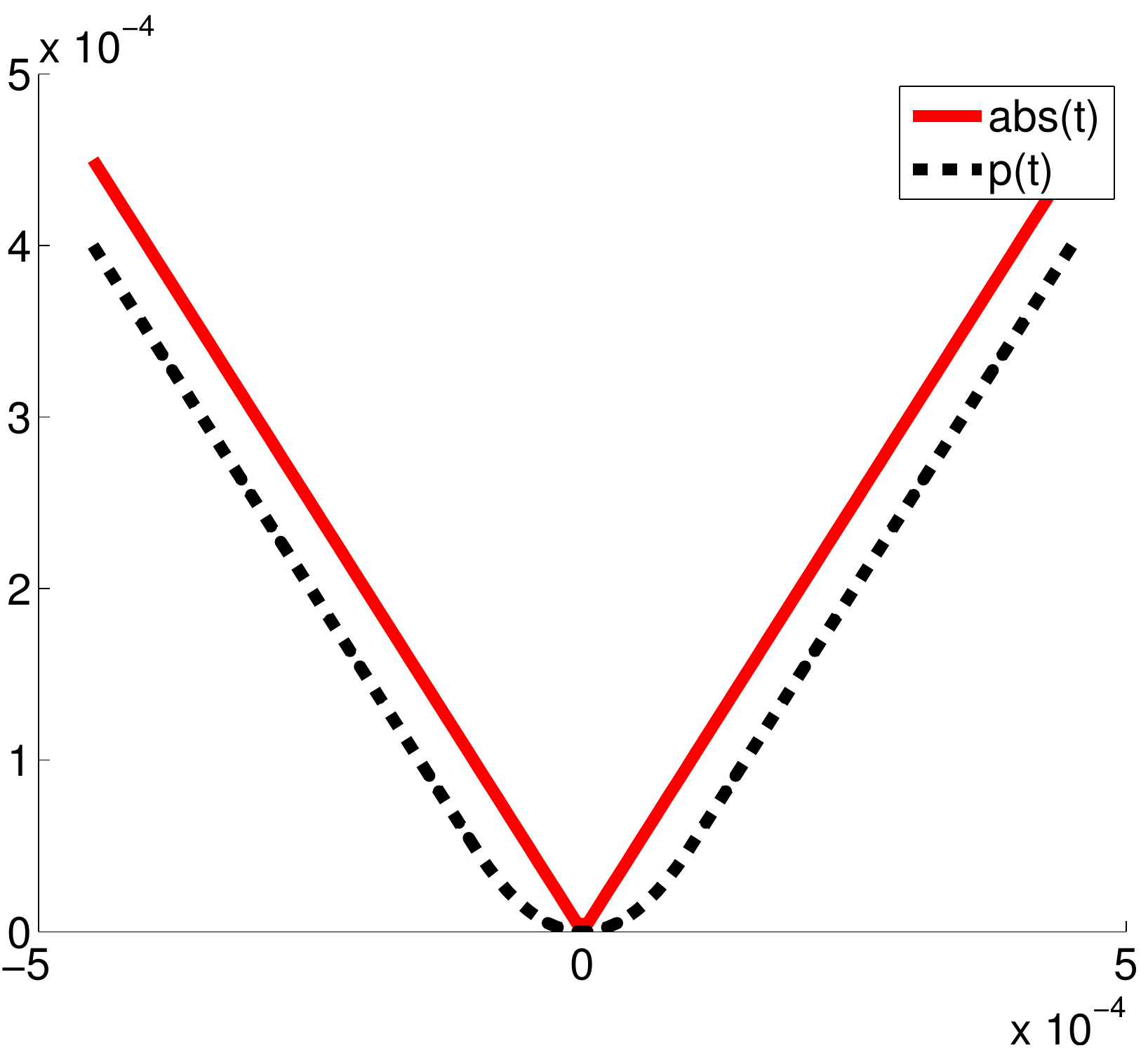} 
}
\caption{
Absolute value function $|t|$ on a fixed 
interval vs different smooth approximations with $\sigma_1 = 3e^{-4}$ 
and $\sigma_2 = e^{-4}$. Row 1: $|t|$ vs $\phi_{\sigma_1}(t), \phi_{\sigma_2}(t)$. 
Row 2: $|t|$ vs $\tilde{\phi}_{\sigma_1}(t), \tilde{\phi}_{\sigma_2}(t)$ and $|t|$ vs $\hat{\phi}_{\sigma_1}(t), 
\hat{\phi}_{\sigma_2}(t)$. 
Row 3: $|t|$ vs $s_{\sigma_1}(t), s_{\sigma_2}(t)$ and $|t|$ vs $p_{\sigma_1}(t), p_{\sigma_2}(t)$. 
\label{fig:smooth_approximations_for_abs_value}}
\end{figure}

\newpage
\section{Gradient Based Algorithms}
In this section, we discuss the use of gradient based optimization algorithms such as 
steepest descent and conjugate gradients to approximately minimize the functional 
\eqref{eq:lp_funct}:
\begin{equation*}
   F_p(x)
   =
   ||Ax - b||^2_2 
     + 2\tau \left( \sum_{k=1}^n |x_k|^p \right)^{1/p},
\end{equation*}
where we make use of one of the approximations ($\phi_{\sigma}(x_k),\tilde{\phi}_{\sigma}(x_k),\hat{\phi}_{\sigma}(x_k)$) from 
\eqref{eq:x_k_approx1},  \eqref{eq:x_k_approx2}, \eqref{eq:x_k_approx3} to replace the non-smooth $|x_k|$. 
Let us first consider the important case of $p=1$ leading to the 
convex $\ell_1$ norm minimization. 
In this case, we approximate the non-smooth functional
\begin{equation*}
   F_1(x) = ||Ax - b||_2^2 + 2\tau ||x||_1
\end{equation*}
by one of the smooth functionals:
\begin{equation}
\label{eq:approx_ell1}
\begin{array}{rcl}
   H_{1,\sigma}(x)
   &=&\displaystyle
   \|Ax-b\|^2_2 + 2\tau\sum_{k=1}^n\phi_{\sigma}(x_k)
\\
   &=&\displaystyle
   \|Ax-b\|^2_2 + 2\tau \displaystyle \sum_{k=1}^n 
      \left( x_k \erf \left( \frac{x_k}{\sqrt{2} \sigma} \right) 
      + \sqrt{\frac{2}{\pi}} \sigma 
      \exp\left(-\frac{x_k^2}{2 \sigma^2}\right) \right)  \\
    
   \tilde{H}_{1,\sigma}(x) &=& \|Ax-b\|^2_2 + 2\tau \displaystyle \sum_{k=1}^n 
      \left( x_k \erf \left( \frac{x_k}{\sqrt{2} \sigma} \right) 
      + \sqrt{\frac{2}{\pi}} \sigma 
      \exp\left(-\frac{x_k^2}{2 \sigma^2}\right) -  \sqrt{\frac{2}{\pi}} \sigma \right) \\

   \hat{H}_{1,\sigma}(x) &=& \|Ax-b\|^2_2 + 2\tau \displaystyle \sum_{k=1}^n 
      \left( x_k \erf \left( \frac{x_k}{\sqrt{2} \sigma} \right) \right) \\
\end{array}
\end{equation}
As with previous approximations, the advantage of working with the smooth 
$H_{1,\sigma}$ functionals instead of $F_1(x)$ is that we can easily compute 
their explicit gradient $\nabla H_{1,\sigma}(x)$ and in this case also the  
Hessian $\nabla^2 H_{1,\sigma}(x)$:

\begin{lemma}
Let $H_{1,\sigma}(x)$, $\tilde{H}_{1,\sigma}(x)$, $\hat{H}_{1,\sigma}(x)$ be as defined in \eqref{eq:approx_ell1} where 
$A\in\real^{m\times n}$, $b\in\real^m$, $\tau,\sigma>0$ 
are constants and 
\[
   \erf(t) := \frac{2}{\sqrt\pi}\int_0^t \exp(-u^2)\, du,\ \ \forall\,
   t\in\real.
\]
Then the gradients are given by:
\begin{eqnarray}
   \nabla H_{1,\sigma}(x) &=& \nabla \tilde{H}_{1,\sigma}(x) =  2A^T(Ax-b) + 2\tau\left\{ 
    \erf\left(\frac{x_k}{\sqrt2 \sigma}\right) \right\}_{k=1}^n \label{eq:Gradient_H1} \\
  \nabla \hat{H}_{1,\sigma}(x) &=& 2A^T(Ax-b) + 2\tau\left\{
      \erf\left(\frac{x_k}{\sqrt2 \sigma}\right) + x_k \frac{1}{\sigma} \sqrt{\frac{2}{\pi}} \exp\left(-\frac{x_k^2}{2 \sigma^2}\right)
    \right\}_{k=1}^n \label{eq:Gradient_hatH1} 
\end{eqnarray}
and the Hessians by:
\begin{eqnarray}
\nabla^2 H_{1,\sigma}(x) &=& \nabla^2 \tilde{H}_{1,\sigma}(x) =  2 A^T A + 
\frac{2\sqrt{2}\tau}{\sigma\sqrt\pi} \Diag\left(\exp\left(-\frac{x_k^2}{2\sigma^2}\right) \right) 
\label{eq:Hessian_H1} \\
\nabla^2 \hat{H}_{1,\sigma}(x) &=&  2 A^T A + 
\frac{4\sqrt{2}\tau}{\sigma\sqrt\pi} \Diag\left(\exp\left(-\frac{x_k^2}{2\sigma^2}\right) - \frac{x_k^2}{2 \sigma^2} \exp\left(-\frac{x_k^2}{2 \sigma^2}\right) \right) \label{eq:Hessian_hatH1}.
\end{eqnarray}
\end{lemma}
where $\Diag:\real^n\to\real^{n\times n}$ 
is a diagonal matrix with the input vector elements on the diagonal.

\begin{proof}
The results follow by direct verification using \eqref{eq:phisigma_derivative} and of the following derivatives:
\begin{eqnarray}
\frac{\dd}{\dd t} \erf\left(\frac{t}{\sqrt{2} \sigma}\right) &=& \frac{\sqrt{2}}{\sigma \sqrt{\pi}} \exp\left(-\frac{t^2}{2 \sigma^2}\right) \label{eq:phiderivstack1} \\
\frac{\dd}{\dd t}  \left[ t \erf\left(\frac{t}{\sqrt{2} \sigma}\right) \right] &=& \erf\left(\frac{t}{\sqrt{2} \sigma}\right) + t \frac{\sqrt{2}}{\sigma \sqrt{\pi}} \exp\left( -\frac{t^2}{2 \sigma^2} \right) \label{eq:phiderivstack2} \\
\frac{\dd^2}{\dd t^2}  \left[ t \erf\left(\frac{t}{\sqrt{2} \sigma}\right) \right] &=& \frac{2 \sqrt{2}}{\sigma \sqrt{\pi}} \exp\left( -\frac{t^2}{2\sigma^2}\right) - t^2 \frac{\sqrt{2}}{\sigma^3 \sqrt{\pi}} \exp\left( -\frac{t^2}{2\sigma^2}\right) \label{eq:phiderivstack3} 
\end{eqnarray}
For instance, using \eqref{eq:phisigma_derivative}, the gradient of $H_{1,\sigma}(x)$ is given by:
\begin{equation*}
\nabla H_{1,\sigma}(x) = \nabla_x \|Ax - b\|_2^2  +  
2\tau \left\{ \frac{\dd}{\dd x_k} \phi_{\sigma}(x_k) \right\}_{k=1}^n = 2 A^T (Ax - b) + 
2\tau \left\{ \erf\left(\frac{x_k}{\sqrt{2} \sigma} \right)  \right\}_{k=1}^n 
\end{equation*}
which establishes \eqref{eq:Gradient_H1}. For the Hessian matrix, we have:
\begin{eqnarray*}
   \nabla^2 H_{1,\sigma}(x) 
   &=& 2A^T A + 2\tau 
  \Diag\left(
     \frac{\dd}{\dd x_1}\erf\left(\frac{x_1}{\sqrt2 \sigma}\right),
     \frac{\dd}{\dd x_2}\erf\left(\frac{x_2}{\sqrt2 \sigma}\right),
     \ldots,
     \frac{\dd}{\dd x_n}\erf\left(\frac{x_n}{\sqrt2 \sigma}\right)
  \right),
\end{eqnarray*}
Using \eqref{eq:phiderivstack1}, we obtain \eqref{eq:Hessian_H1}. Similar computations 
using \eqref{eq:phiderivstack2} and \eqref{eq:phiderivstack3} yield \eqref{eq:Gradient_hatH1} 
and \eqref{eq:Hessian_hatH1}.
\end{proof}

Next, we discuss the smooth approximation to the general functional  
\eqref{eq:lp_funct}. In particular, we are interested in the case $p<1$. 
In this case, the functional is not convex, but may still be useful in 
compressive sensing applications \cite{MR2421974}. We presents the results using 
the approximation $\phi_\sigma$ from \eqref{eq:phisigma}. The calculations with 
$\tilde{\phi}_\sigma$ and $\hat{\phi}_{\sigma}$ take similar form. We obtain the 
approximation functional to $F_p(x)$:
\begin{equation}
\label{eq:approx_lp}
\begin{array}{rcl}
   H_{p,\sigma}(x) 
   &:=&\displaystyle
   \|Ax - b\|^2_2 + 2\tau \left( \sum_{k=1}^n 
      \phi_\sigma(x_k)^p \right)^{1/p} 
\\
   &=&\displaystyle
   \|Ax - b\|^2_2 + 2\tau \left( \sum_{k=1}^n 
      \left( x_k \erf \left( \frac{x_k}{\sqrt{2} \sigma} \right) 
      + \sqrt{\frac{2}{\pi}} \sigma 
      \exp\left(\frac{-x_k^2}{2 \sigma^2}\right)\right)^p \right)^{1/p} .
\end{array}
\end{equation}

\begin{lemma}
Let $H_{p,\sigma}(x)$ be as defined in \eqref{eq:approx_lp} where $p>0$ and $\sigma > 0$.
Then the gradient is given by:
\begin{equation}
\label{eq:lp_gradient}
   \nabla {H_{p,\sigma}}(x)
   = 
   2 A^T (Ax - b) + 
   2\tau p \left( \sum_{k=1}^n \phi_{\sigma}(x_k)^p \right)^{(1-p)/p} 
   \left\{ \phi_{\sigma}(x_j)^{p-1} 
   \erf\left(\frac{x_j}{\sqrt{2} \sigma}\right) \right\}_{j=1}^n, 
\end{equation}
and the Hessian is given by:
\begin{equation}
\label{eq:lp_hessian}
   \nabla^2 H_{p,\sigma}(x) = 2 A^T A + 2\tau 
      \left(v(x) v(x)^T + \Diag\bigl(w(x)\bigr)\right),
\end{equation}
where the functions $v,w:\real^n\to\real^n$ are defined for all $x\in\real^n$:
\begin{small}
\begin{align*}
   v(x):=&\ 
   \left\{
      \sqrt{1-p}\, \phi_\sigma(x_j)^{p-1} 
      \erf\left(\frac{x_j}{\sqrt{2}\sigma}\right)
      \left(\sum_{k=1}^n \phi_{\sigma}(x_k)^p\right)^{(1-2p)/(2p)}
   \right\}_{j=1}^n,
\\
   w(x):=&\ 
   \left(\sum_{k=1}^n \phi_{\sigma}(x_k)^p\right)^{(1-p)/p} \left\{
      (p-1)\phi_\sigma(x_j)^{p-2}\left(\erf\left(\frac{x_j}{\sqrt{2}\sigma}\right)\right)^2
       +\frac{\sqrt{2}}{\sigma \sqrt{\pi}}\phi_\sigma(x_j)^{p-1}
       \exp\left(-\frac{x_j^2}{2 \sigma^2}\right)
   \right\}_{j=1}^n.
\end{align*}
\end{small}
\end{lemma}

\begin{proof}
Define $G_{p,\sigma}:\real^n\to\real$ by
\begin{equation}
\label{eq:Gpsigma}
   G_{p,\sigma}(x)
   :=\left(\sum_{k=1}^n \phi_{\sigma}(x_k)^p\right)^{1/p}.
\end{equation}
Then $H_{p,\sigma}(x)=\|Ax-b\|+2\tau G_{p,\sigma}(x)$ for all $x$, 
and for each $j=1,\ldots,n$,
\begin{equation}
\label{eq:Gpartial1}
   \frac{\partial}{\partial x_j}G_{p,\sigma}(x)
   =
   \frac{1}{p} \left(\sum_{k=1}^n \phi_{\sigma}(x_k)^p\right)^{(1-p)/p}
      \left(p\phi_\sigma(x_j)^{p-1}\phi_\sigma'(x_j)\right)
   =
    G_{p,\sigma}(x)^{1-p} \phi_\sigma(x_j)^{p-1} 
      \erf\left(\frac{x_j}{\sqrt{2}\sigma}\right)
\end{equation}
where we have used \eqref{eq:phisigma_derivative}. Hence, 
\eqref{eq:lp_gradient} follows.
Next, we compute the Hessian of $G_{p,\sigma}$.
For each $i\neq j$,
\begin{eqnarray*}
   \frac{\partial^2}{\partial x_i\partial x_j}G_{p,\sigma}(x)
   &=& \frac{\partial}{\partial x_i} \left[ G_{p,\sigma}(x)^{1-p} \phi_\sigma(x_j)^{p-1} \erf\left(\frac{x_j}{\sqrt{2}\sigma}\right)  \right] =  \phi_\sigma(x_j)^{p-1} \erf\left(\frac{x_j}{\sqrt{2}\sigma}\right)  \frac{\partial}{\partial x_i} \left[ G_{p,\sigma}(x)^{1-p} \right] \\ 
 &=&  (1-p) \phi_\sigma(x_j)^{p-1} 
      \erf\left(\frac{x_j}{\sqrt{2}\sigma}\right)
      G_{p,\sigma}(x)^{-p}\frac{\partial}{\partial x_i}G_{p,\sigma}(x) \\
   &=& 
   (1-p) \phi_\sigma(x_i)^{p-1} \phi_\sigma(x_j)^{p-1} 
      \erf\left(\frac{x_i}{\sqrt{2}\sigma}\right)
      \erf\left(\frac{x_j}{\sqrt{2}\sigma}\right)
      G_{p,\sigma}(x)^{1-2p} \\
   &=& v(x)_i v(x)_j = \left(v(x) v(x)^T\right)_{ij},
\end{eqnarray*}
and when $i=j$, for each $j=1,\ldots,n$,
\begin{eqnarray*}
   \frac{\partial^2}{\partial^2 x_j}G_{p,\sigma}(x) &=& 
    \frac{\partial}{\partial x_j} \left[ G_{p,\sigma}(x)^{1-p} \phi_\sigma(x_j)^{p-1} \erf\left(\frac{x_j}{\sqrt{2}\sigma}\right) \right] \\
   &=& 
   (1-p) \phi_\sigma(x_j)^{2(p-1)} 
      \left(\erf\left(\frac{x_j}{\sqrt{2}\sigma}\right)\right)^2
      G_{p,\sigma}(x)^{1-2p} \\
    &+&
   G_{p,\sigma}(x)^{1-p} \left(
      (p-1)\phi_\sigma(x_j)^{p-2}\left(\erf\left(\frac{x_j}{\sqrt{2}\sigma}\right)\right)^2
       +\frac{\sqrt{2}}{\sigma \sqrt{\pi}}\phi_\sigma(x_j)^{p-1}
       \exp\left(-\frac{x_j^2}{2 \sigma^2}\right)
   \right)
\\ 
   &=& 
   v(x)_j^2 + w(x)_j
   = \bigl(v(x)v(x)^T + \Diag\bigl(w(x)\bigr)\bigr)_{jj}.
\end{eqnarray*}
Hence, \eqref{eq:lp_hessian} holds.
\end{proof}

\newpage
Given $H_{p,\sigma}(x) \approx F_p(x)$ and $\nabla H_{p,\sigma}(x)$, we can apply a number of gradient based methods for the minimization of $H_{p,\sigma}(x)$ (and hence for the approximate minimization 
of $F_p(x)$), which take the following general form:

\begin{algorithm}[ht!]
\SetKwInOut{Input}{Input}
\SetKwInOut{Output}{Output}
\caption{Generic Gradient Method for finding $\arg\min H_{p,\sigma}(x)$.}
\label{algo:gradient_method}
\BlankLine
Pick an initial point $x^0$\;
\For{$k=0,1,\ldots$,\texttt{\textup{ maxiter}}}{
Compute search direction $s^n$ based on gradient 
$\nabla H_{p,\sigma}(x^n)$. \;
Compute step size parameter $\mu$ via line search. \;
Update the iterate: $x^{n+1} = x^n + \mu s^n$. \;
Check if the termination conditions are met. \;
}
Record final solution: $\bar{x} = x^{n+1}$. \;
\end{algorithm}
Note that in the case of $p<1$, the functional $F_p(x)$ is not convex, so such an algorithm may 
not converge to the global minimum in that case. The generic algorithm above depends on the choice 
of search direction $s^n$, which is based on the gradient, and the line search, 
which can be performed several different ways. 

\noindent
\subsection{Line Search Techniques}
Gradient based algorithms differ based on 
the choice of search direction vector $s^n$ and 
line search techniques for parameter $\mu$. 
In this section we describe some suitable line search 
techniques. 
Given the current iterate $x^n$ and search direction $s^n$, 
we would like to choose $\mu$ so that:
\begin{equation*} 
   H_{p,\sigma}(x^{n+1})
   = H_{p,\sigma}(x^n + \mu s^n)
   \leq H_{p,\sigma}(x^n),
\end{equation*}
where $\mu>0$ is a scalar which measures how long along 
the search direction we advance from the previous iterate. 
Ideally, we would 
like a strict inequality and the functional value to decrease. 
Exact line search would solve the single variable 
minimization problem:
\begin{equation*}
   \bar{\mu} = \arg\min_{\mu} H_{p,\sigma}(x^n + \mu s^n) .
\end{equation*}
The first order necessary optimality condition
(i.e., $\nabla H_{p,\sigma}(x + \mu s)^T s=0$) can be used to find a 
candidate value for $\mu$, but it is not easy to solve the gradient equation. 

An alternative approach is to use a backtracking line search  
to get a step size $\mu$ that satisfies one or two of 
the Wolfe conditions \cite{nocedal_wright_opt} as in Algorithm \ref{alg:wolfe_line_search}. 
This update scheme can be slow since several evaluations 
of $H_{p,\sigma}(x)$ may be necessary, which are relatively expensive when the 
dimension $n$ is large. It also depends on the choice of parameters $\rho$ and $c$, 
to which the generic gradient method may be sensitive.

\newpage

\begin{algorithm}
\SetKwInOut{Input}{Input}
\SetKwInOut{Output}{Output}
\caption{Backtracking Line Search \label{alg:wolfe_line_search}}
\Input{Evaluators for $H_{p,\sigma}(x)$ and $\nabla H_{p,\sigma}(x)$,
current iterate $x^n$, search direction $s^n$, 
and constants $\mu > 0$, $\rho \in (0,1)$, $c \in (0,1)$.}
\Output{$\mu>0$ satisfying a sufficient decrease condition.}
\BlankLine
\While{$H_{p,\sigma}(x^n + \mu s^n) > H_{p,\sigma}(x^n) 
   + c \mu (\nabla H_{p,\sigma}(x^n))^T s^n$ }{
$\mu = \rho \mu$ \;
}
\end{algorithm}

\noindent
Another way to perform approximate line search is to utilize a 
Taylor series approximation for the solution of 
$\frac{\dd}{\dd \mu} H_{p,\sigma}(x + \mu s) = 0$ \cite{shewchukCG}. 
This involves the gradient and Hessian terms which we have 
previously computed. 
Using the second order Taylor approximation of $n(t):=H_{p,\sigma}(x+t s)$ at any given $x,s\in\real^n$, 
we have that 
\begin{equation}
\label{eq:taylor_expansion_of_H}
   n^{\prime}(t)=n^{\prime}(0)+t n^{\prime \prime}(0)+o(t) \approx n^{\prime}(0)+t n^{\prime \prime}(0)
\end{equation}
using basic matrix calculus:
\begin{eqnarray*}
   n^{\prime}(t) 
   &=& 
   \left( \nabla H_{p,\sigma}(x+ts) \right)^T s 
   \implies 
   n^{\prime}(0) = \nabla H_{p,\sigma}(x)^T s 
\\
   n^{\prime \prime}(t) 
   &=& \left[ \left( \nabla^2 H_{p,\sigma}(x + ts) \right)^T s \right]^T s =  
   s^T \nabla^2 H_{p,\sigma}(x + ts) s 
   \implies 
   n^{\prime \prime}(0) = s^T \nabla^2 H_{p,\sigma}(x) s,
\end{eqnarray*}
we get that $n^{\prime}(0)+\mu n^{\prime \prime}(0)=0$ if and only if
\begin{equation}
\label{eq:mu_approx_sol}
\mu = -\frac{\nabla H_{p,\sigma}(x)^T s}{s^T \nabla^2 H_{p,\sigma}(x) s},
\end{equation}
which can be used as the step size in Algorithm \ref{algo:gradient_method}.

For the case $p=1$ (approximating the $\ell_1$ functional), 
the Hessian is $A^T A$ plus a diagonal matrix, 
which is quick to form and the above approximation can be 
efficiently used for line search. 
For $p\neq1$, the Hessian is the sum of $A^T A$ and $M$, 
and $M$ in turn is the sum of a diagonal matrix and a rank one matrix; 
the matrix-vector multiplication involving this Hessian is 
more expensive than in the case $p=1$.
In this case, one may approximate the Hessian in 
\eqref{eq:mu_approx_sol} using finite differences, i.e., 
when $\xi > 0$ is sufficiently small,
\begin{equation}
\label{eq:second_deriv_secant}
   n^{\prime \prime}(t) \approx \frac{n^{\prime}(t + \xi) - n^{\prime}(t - \xi)}{2\xi} .
   \implies n^{\prime \prime}(0) \approx \frac{n^{\prime}(\xi) - n^{\prime}(-\xi)}{2\xi}
\end{equation}
Approximating $n^{\prime\prime}(0)$ in \eqref{eq:taylor_expansion_of_H} by 
$\frac{n^{\prime}(\xi) - n^{\prime}(-\xi)}{2\xi}$, we get
\begin{equation}
\label{eq:secant_approx_n}
   \frac{\dd}{\dd\mu} H_{p,\sigma}(x + \mu s) 
   \approx \nabla H_{p,\sigma}(x)^Ts 
      + \mu \frac{\left(\nabla H_{p,\sigma}(x + \xi s) 
      - \nabla H_{p,\sigma}(x - \xi s)\right)^T s}{2\xi}.
\end{equation}
Setting the right hand side of \eqref{eq:secant_approx_n} to zero,
and solving for $\mu$, we get the approximation:
\begin{equation}
\label{eq:mu_approx_sol2}
   \mu 
   = 
   \frac{-2 \xi\nabla H_{p,\sigma}(x)^T s}
      {(\nabla H_{p,\sigma}(x + \xi s) - \nabla H_{p,\sigma}(x - \xi s))^T s}.
\end{equation}
In the finite difference scheme, the parameter $\xi$ should be taken to be 
of the same order as the components of the current iterate $x^n$. 
In practice, we find that for $p=1$, the Hessian based line search \eqref{eq:mu_approx_sol} 
works well; for $p\neq1$, one can also use the finite difference scheme 
\eqref{eq:mu_approx_sol2} if one wants to avoid evaluating the Hessian. 

\vspace{3.mm}
\subsection{Steepest Descent and Conjugate Gradient Algorithms}
We now present steepest descent and conjugate gradient schemes, 
in Algorithms \ref{algo:steepestdescent} and \ref{algo:nonlincg} respectively, 
which can be used for sparsity constrained regularization. We also discuss the use 
of Newton's method in Algorithm \ref{algo:nonlin_newton}. 

Steepest descent and conjugate gradient methods differ in the choice of the search direction. 
In steepest descent methods, we simply take the negative of the gradient 
as the search direction. For nonlinear conjugate gradient methods, 
which one expects to perform better than steepest descent, 
several different search direction updates are possible. 
We find that the Polak-Ribi\`{e}re scheme often offers good performance 
\cite{MR0255025,Polyak196994,shewchukCG}. 
In this scheme, we set the initial search direction $s^0$  
to the negative gradient, as in steepest descent, but then do a more 
complicated update involving the gradient at the current and previous steps:
\begin{eqnarray*}
   \beta^{n+1} 
   &=& 
   \max\left\{ \frac{\nabla H_{p,\sigma_n}(x^{n+1})^T 
      \left(\nabla H_{p,\sigma_n}(x^{n+1}) - \nabla H_{p,\sigma_n}(x^n)\right)}
      {\nabla H_{p,\sigma_n}(x^n)^T \nabla H_{p,\sigma_n}(x^n)}, 0 \right\} ,
\\
   s^{n+1} 
   &=& -\nabla H_{p,\sigma_n}(x^{n+1}) + \beta^{n+1} s^n.
\end{eqnarray*}

One extra step we introduce in Algorithms \ref{algo:steepestdescent} and
\ref{algo:nonlincg} below is a thresholding 
which sets small components to zero. That is, at the end of each iteration, we 
retain only a portion of the largest coefficients. This is necessary, 
as otherwise the solution we recover will contain many small noisy components 
and will not be sparse. In our numerical experiments, we found that soft 
thresholding works well when $p=1$ and that hard thresholding works well 
when $p<1$. 
The componentwise soft and hard thresholding functions with parameter 
$\tau>0$ are given by: 
\begin{equation}
\label{eq:thresholding}
   \left(\mathbb{S}_{\tau }(x)\right)_k = \left\{
   \begin{array}{ll}
       x_k - \tau, & \hbox{$x_k > \tau$} \\
       0, & \hbox{$-\tau \le x_k \le \tau$;} \\
       x_k + \tau, & \hbox{$x_k < -\tau$ } \\
   \end{array}
   \right. \quad \left(\mathbb{H}_{\tau }(x)\right)_k = \left\{
   \begin{array}{ll}
       x_k, & \hbox{$|x_k| > \tau$} \\
       0, & \hbox{$-\tau \le x_k \le \tau$}
   \end{array},
   \right. 
\quad\forall\, x\in\real^n.
\end{equation}
For $p=1$, an alternative to thresholding at each iteration at $\tau$ is to 
use the optimality condition of the $F_1(x)$ functional \cite{ingrid_thresholding1}. 
After each iteration (or after a block of iterations), we can evaluate the vector  
\begin{equation}
\label{eq:v_thresholding}
v^n = A^T (b - A x^n) . 
\end{equation}
We then set the components (indexed by $k$) of the current solution vector 
$x^n$ to zero for indices $k$ for which $|v^n_k| \leq \tau$. 

Note that after each iteration, we also vary the parameter $\sigma$ in the 
approximating function to the absolute value $\phi_\sigma$, 
starting with $\sigma$ relatively far from zero at the first iteration 
and decreasing towards $0$ as we approach the iteration limit. 
The decrease can be controlled by a parameter $\alpha\in(0,1)$ 
so that $\sigma_{n+1} = \alpha \sigma_n$. 
The choice $\alpha = 0.8$ worked well in our experiments. 
We could also tie $\sigma^n$ to the progress of the iteration,
such as the quantity $||x^{n+1} - x^n||_2$. 
One should experiment to find what works best with a given application. 

Finally, we comment on the computational cost of Algorithms 
\ref{algo:steepestdescent} and \ref{algo:nonlincg},
relative to standard iterative thresholding methods, notably the FISTA method.
The FISTA iteration for $F_1(x)$, for example, would be implemented as:
\begin{equation}
\label{eq:FISTA}
   y^{0} = x^{0} 
   \mbox{ ,} \quad 
   x^{n+1} = S_{\tau}\left(y^n + A^T b - A^T A y^n\right) 
   \mbox{ ,} \quad 
   y^{n+1} = x^{n+1} + \frac{t_k - 1}{t_{k+1}}\left(x^{n+1} - x^{n}\right),
\end{equation}
where $\{t_k\}$ is a special sequence of constants \cite{MR2486527}. For large linear systems, 
the main cost is in the evaluation of $A^T A y^n$. The same is true for the gradient based schemes we 
present below. The product of $A^T A$ and the vector iterate goes into the gradient computation 
and the line search method and can be shared between the two. 
Notice also that the gradient and line search computations involve the evaluation of the error function 
$\erf(t) = \frac{2}{\sqrt{\pi}} \int_{0}^{t} e^{-u^2} du$,
and there is no closed form solution for this integral.
However, various ways of efficiently approximating the integral value exist:
apart from standard quadrature methods, 
several approximations involving the exponential function 
are described in \cite{HandbookOfMathFunctions}.
The gradient methods below do have extra overhead compared to the thresholding 
schemes and may not be ideal for runs with large numbers of iterations. 
However, for large matrices and with efficient implementation, the runtimes for our schemes 
and existing iterative methods are expected to be competitive,
since the most time consuming step (multiplication with $A^T A$) is common to both.

Algorithm \ref{algo:steepestdescent}, below, presents a simple steepest descent scheme 
to approximately minimize $F_p$ defined in \eqref{eq:lp_funct}. 
In Algorithm \ref{algo:nonlincg}, we present a nonlinear Polak-Ribi\`{e}re conjugate 
gradient scheme to approximately minimize $F_p$ \cite{MR0255025,Polyak196994,shewchukCG}.
In practice, this slightly more complicated algorithm is expected to perform 
significantly better than the simple steepest descent method.

Another possibility, given access to both gradient and Hessian, is to use a higher 
order root finding method, such as Newton's method \cite{nocedal_wright_opt} presented 
in Algorithm \ref{algo:nonlin_newton}. The idea here is to find a 
root of $\nabla H_{p,\sigma}(x) = 0$, given some initial guess $x^{0}$ which should 
correspond to a local extrema of $H_{p,\sigma}$. By classical 
application of Newton's method for vector valued functions, we obtain the simple 
scheme: $x^{n+1} = x^n + \Delta x$ with $\Delta x$ the solution to the linear 
system $\nabla^2 H_{p,\sigma_n}(x^n) \Delta x = - \nabla H_{p,\sigma_n}(x^n)$.
However, Newton's method usually requires an accurate initial guess $x^0$ 
\cite{nocedal_wright_opt}. For this reason, the presented scheme would 
usually be used to top off a CG algorithm or sandwiched between 
CG iterations. 

The function $\mathtt{Threshold}(\cdot,\tau)$ in the algorithms which enforces sparsity refers to either one of the two thresholding functions defined in \eqref{eq:thresholding} or to the strategy using the $v^n$ vector in \eqref{eq:v_thresholding}.

\newpage

\begin{algorithm}[!ht]
\SetKwInOut{Input}{Input}
\SetKwInOut{Output}{Output}
\caption{Steepest Descent Scheme
\label{algo:steepestdescent}}
\Input{An $m\times n$ matrix $A$, an initial guess $n \times 1$ vector $x^0$, 
a parameter $\tau < \|A^T b\|_{\infty}$, 
a parameter $p\in(0,1]$, 
a parameter $\sigma_0>0$, 
a parameter $0< \alpha < 1$, 
the maximum number of iterations $M$, 
and a routine to evaluate the gradient $\nabla H_{p,\sigma}(x)$ (and possibly 
the Hessian $\nabla^2 H_{p,\sigma}(x)$ depending on choice of line search method).}
\Output{A vector $\bar{x}$, 
close to either the global or local minimum of $F_p(x)$, 
depending on choice of $p$.}
\BlankLine
\For{$k=0,1,\ldots$,M}{
$s^n = -\nabla H_{p,\sigma_n}(x^n)$ \; 
use line search to find $\mu>0$\;
$x^{n+1} = \mathtt{Threshold}(x^n + \mu s^n, \tau)$ \;
$\sigma_{n+1} = \alpha \sigma_n$ \;
}
$\bar{x} = x^{n+1}$\;
\end{algorithm}

\begin{algorithm}[!ht]
\SetKwInOut{Input}{Input}
\SetKwInOut{Output}{Output}
\caption{Nonlinear Conjugate Gradient Scheme
\label{algo:nonlincg}}
\Input{An $m\times n$ matrix $A$, 
an initial guess $n \times 1$ vector $x^0$, 
a parameter $\tau < \|A^T b\|_\infty$, 
a parameter $p\in(0,1]$, 
a parameter $\sigma_0>0$, 
a parameter $0 < \alpha < 1$, 
the maximum number of iterations $M$, 
and a routine to evaluate the gradient $\nabla H_{p,\sigma}(x)$ (and possibly 
the Hessian $\nabla^2 H_{p,\sigma}(x)$ depending on choice of line search method).}
\Output{A vector $\bar{x}$, 
close to either the global or local minimum of $F_p(x)$, 
depending on choice of $p$.}
\BlankLine
$s^0 = - \nabla H_{p,\sigma_0}(x^0)$ \;
\For{$k=0,1,\ldots$,M}{
use line search to find $\mu>0$\;
$x^{n+1} = \mathtt{Threshold}(x^n + \mu s^n, \tau)$ \;
$\beta^{n+1} = 
\max\left\{ \frac{\nabla H_{p,\sigma_n}
(x^{n+1})^T (\nabla H_{p,\sigma_n}(x^{n+1}) - \nabla H_{p,\sigma_n}(x^n))}
{\nabla H_{p,\sigma_n}(x^n)^T \nabla H_{p,\sigma_n}(x^n)}, 0 \right\}$ \;
$s^{n+1} = -\nabla H_{p,\sigma_n}(x^{n+1}) + \beta^{n+1} s^n$ \;
$\sigma_{n+1} = \alpha \sigma_n$ \;
}
$\bar{x} = x^{n+1}$\;
\end{algorithm}

\begin{algorithm}[!ht]
\SetKwInOut{Input}{Input}
\SetKwInOut{Output}{Output}
\caption{Newton's Method
\label{algo:nonlin_newton}}
\Input{An $m\times n$ matrix $A$, 
a parameter $\tau < \|A^T b\|_\infty$, 
a parameter $p\in(0,1]$, 
a parameter $\sigma_0>0$, 
a parameter $0 < \alpha < 1$, 
the maximum number of iterations $M$, 
a tolerance parameter TOL,
and routines to evaluate the gradient $\nabla H_{p,\sigma}(x)$ and  
the Hessian $\nabla^2 H_{p,\sigma}(x)$.}
\Output{A vector $\bar{x}$, close to either the global or local minimum of $F_p(x)$, 
depending on choice of $p$.}
\BlankLine
Obtain a relatively accurate initial guess $x^0$. \;
\For{$k=0,1,\ldots$,M}{
Solve the linear system 
$\nabla^2 H_{p,\sigma_n}(x^n) \Delta x = - \nabla H_{p,\sigma_n}(x^n)$ to tolerance (TOL). \;
$x^{n+1} = \mathtt{Threshold}(x^n + \Delta x, \tau)$ \;
$\sigma_{n+1} = \alpha \sigma_n$ \;
}
$\bar{x} = x^{n+1}$\;
\end{algorithm}

\newpage
\section{Numerical Experiments}
\label{sect:numerics}

We now show some numerical experiments, comparing Algorithm \ref{algo:nonlincg}  
with FISTA, a state of the art sparse regularization algorithm \cite{MR2486527}
outlined in \eqref{eq:FISTA}. We use our CG scheme with $p=1$ and later also 
with Newton's method (Algorithm \ref{algo:nonlin_newton}) and with $p<1$. 
We present plots of averaged quantities over many trials, wherever possible. 
We observe that for these experiments, the CG scheme gives good results in few iterations, 
although each iteration of CG is more expensive than a single iteration of FISTA. 
To account for this, we run the experiments using twice more iterations for FISTA ($100$) 
than for CG ($50$). The Matlab codes for all the experiments and figures described and 
plotted here are available for download from the author's website. 

When performing a sparse reconstruction, we typically vary 
the value of the regularization parameter $\tau$ and move along a regularization parameter 
curve while doing warm starts, starting from a relatively high value of $\tau$ close 
to $||A^T b||_{\infty}$ with a zero initial guess  (since for $\tau > ||A^T b||_{\infty}$, the 
$\ell_1$ minimizer is zero \cite{ingrid_thresholding1}) and moving to a lower value,  
while reusing the solution at the previous $\tau$ as the initial 
guess at the next, lower $\tau$ \cite{Loris2009247, Sindhwani2012}. If the $\tau$'s follow a 
logarithmic decrease, the corresponding curves we observe are 
like those plotted in Figure \ref{fig:sparse_reconstruction2}. 
At some $\tau$, the reconstruction $x_{\tau}$ will be optimal along the curve and the 
percent error between solution $x_{\tau}$ and true solution $x$ will be lowest. If we do not 
know the true solution $x$, we have to use other criteria to pick the $\tau$ at which 
we want to record the solution. One way is by using the norm of the noise vector in the 
right hand side $||e||_2$. If an accurate estimate of this is known, we can use the 
solution at the $\tau$ for which the residual norm $||A x_{\tau} - b||_2 \approx ||e||_2$. 

For our examples in Figure \ref{fig:sparse_reconstruction2} and 
\ref{fig:sparse_reconstruction3}, we use three types of 
matrices, each of size $1000 \times 1000$. We use random Gaussian matrices constructed 
to have fast decay of singular values 
(matrix type I), matrices with a portion of columns which are linearly correlated 
(matrix type II - formed by taking matrix type I and forcing a random sample 
of 200 columns to be approximately linearly dependent with some of the others), 
and matrices with entries from the random Cauchy distribution \cite{MR1326603} 
(matrix type III). For our CG scheme, we 
use the scheme as presented in Algorithm \ref{algo:nonlincg} with the approximation 
$\hat{\phi}_{\tau}(x)$ for $|x|$. 

In Figure \ref{fig:sparse_reconstruction2}, we plot the residuals vs $\tau$ for 
the two algorithms. We also plot curves for the percent errors 
$100 \frac{||x_{\tau} - x||_2}{||x||_2}$ and the functional values 
$F_1(x_{\tau})$ vs $\tau$ (note that CG is in fact minimizing an approximation 
to the non-smooth $F_1(x)$, yet for these examples we find that the value of $F_1(x)$ 
evaluated for CG is often lower than for FISTA, even when FISTA is run at twice more iterations). 
The curves shown are median values recorded over $20$ runs.

We present more general contour plots in Figure \ref{fig:sparse_reconstruction3} which 
compare the minimum percent errors along the regularization curve produced by the two algorithms at different combinations of number of nonzeros and noise levels.  
The data at each point of the contour plots 
is obtained by running FISTA for $100$ iterations and CG for $50$ iterations at each $\tau$ 
starting from $\tau = \frac{||A^T b||_{\infty}}{10}$ going down to $\tau = \frac{||A^T b||_{\infty}}{5e^{8}}$ and reusing the previous solution as the initial guess at the next $\tau$. 
We do this for $10$ trials and record the median values.  

From Figures \ref{fig:sparse_reconstruction2} and \ref{fig:sparse_reconstruction3}, we observe 
similar performance of CG and FISTA for matrix type I and significantly better performance of 
CG for matrix types II and III where FISTA does not do as well. 

In Figure \ref{fig:sparse_reconstruction4}, we do a compressive sensing image recovery 
test by trying to recover the original image from its samples. Here we also test 
CG with Newton's method and CG with $p < 1$. A sparse image $x$ was used 
to construct the right hand side with a sensing matrix $A$ via $b = A \tilde{x}$ 
where $\tilde{x}$ is the noisy image, with $\tilde{x} = x + e$ and $e$ being the noise 
vector with $25$ percent noise level relative to the norm of $x$. 
The matrix $A$ was constructed to be as matrix type II described above. 
The number of columns and image pixels was $5025$. The number of rows 
(image samples) is $5000$. We run the algorithms to recover an approximation to $x$ 
given the sensing matrix $A$ and the noisy 
measurements $b$. For each algorithm we used a fixed number of iterations at each $\tau$ along 
the regularization curve as before, from $\tau = \frac{||A^T b||_{\infty}}{10}$ going down 
to $\tau = \frac{||A^T b||_{\infty}}{5e^{8}}$ and reusing the previous solution as 
the initial guess at the next $\tau$ starting from a zero vector guess at the beginning. 
Each CG algorithm is run for a total of 50 iterations at each $\tau$ and FISTA for 100 iterations. 
The first CG method uses 50 iterations with $p=1$ using Algorithm \ref{algo:nonlincg}. 
The second CG method uses $30$ 
iterations with $p=1$, followed by Newton's method for $5$ iterations (using 
Algorithm \ref{algo:nonlin_newton} with the system solve done via CG for 15 iterations at each 
step) and a further $15$ iterations of CG with the initial guess from the result of the Newton scheme. 
That is, we sandwich $5$ Newton iterations within the CG scheme.  
The final CG scheme uses $50$ iterations of CG with $p=0.83$ (operating on the 
non-convex $F_p(x)$ for $p<1$). 
In these experiments, we used Hessian based line search approximation \eqref{eq:mu_approx_sol} and 
soft thresholding \eqref{eq:thresholding} for $p=1$ and the finite difference line 
search approximation \eqref{eq:mu_approx_sol2} and hard thresholding 
\eqref{eq:thresholding} for $p<1$. 
The results are shown in Figure \ref{fig:sparse_reconstruction4}. 
We observe that from the same number of samples, better reconstructions 
are obtained using the CG algorithm and that $p$ slightly less than $1$ can give 
even better performance than $p=1$ in terms of recovery error. Including a few 
iterations of Newton's scheme also seems to slightly improve the result. All of the 
CG schemes demonstrate better recovery than FISTA in this test.

\begin{figure*}[ht!]
\centerline{
\includegraphics[scale=0.25]{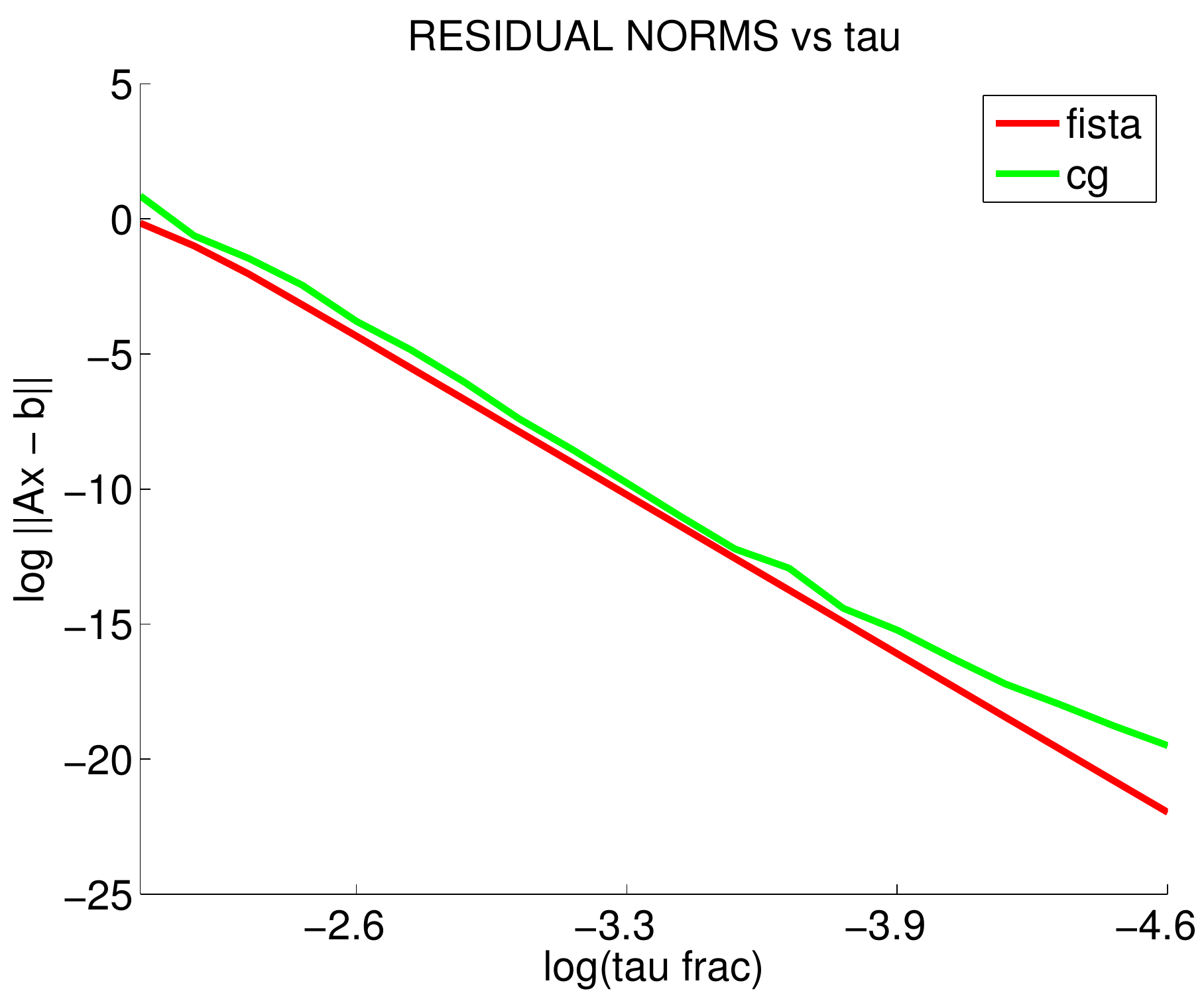}
\includegraphics[scale=0.25]{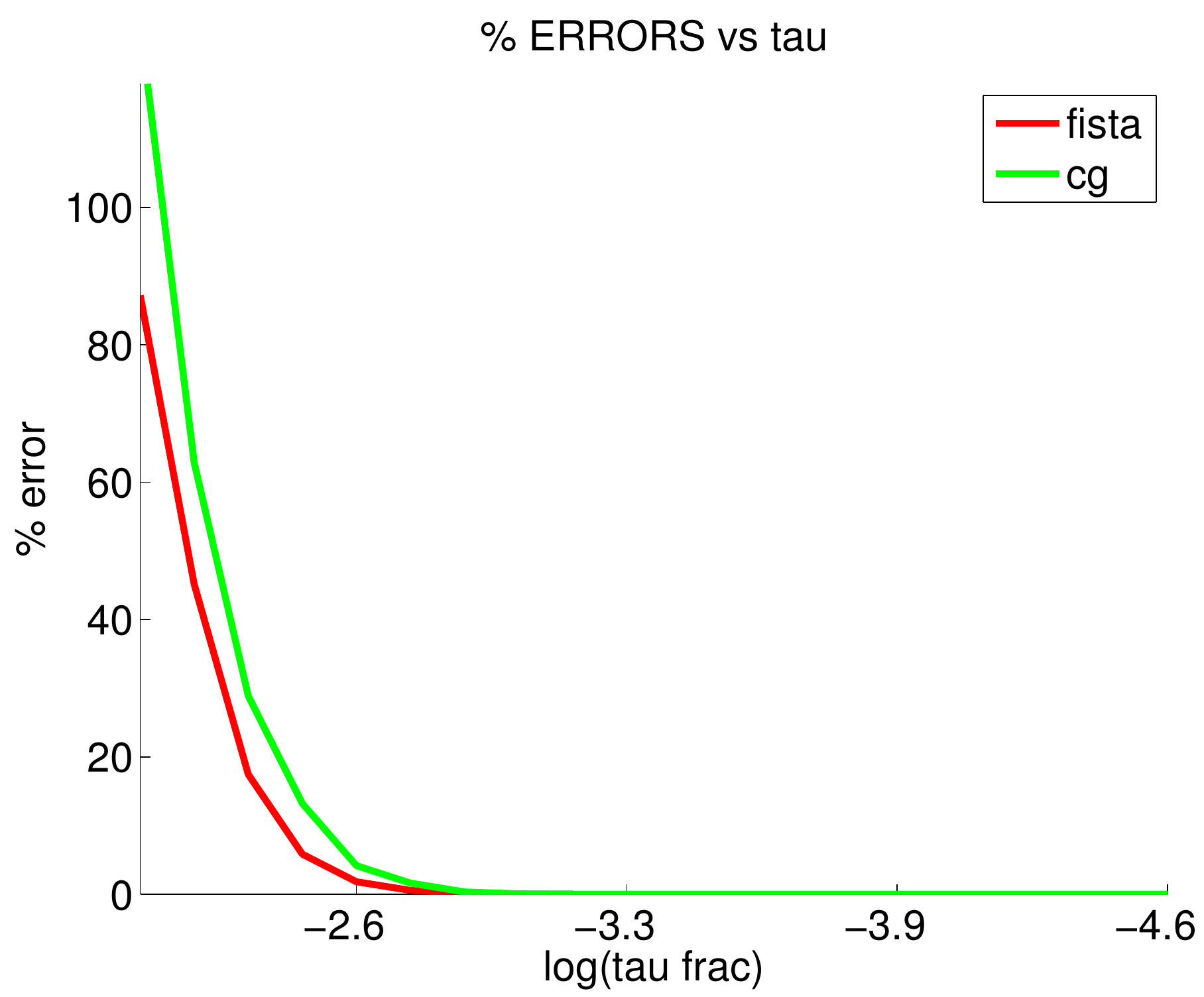}
\includegraphics[scale=0.25]{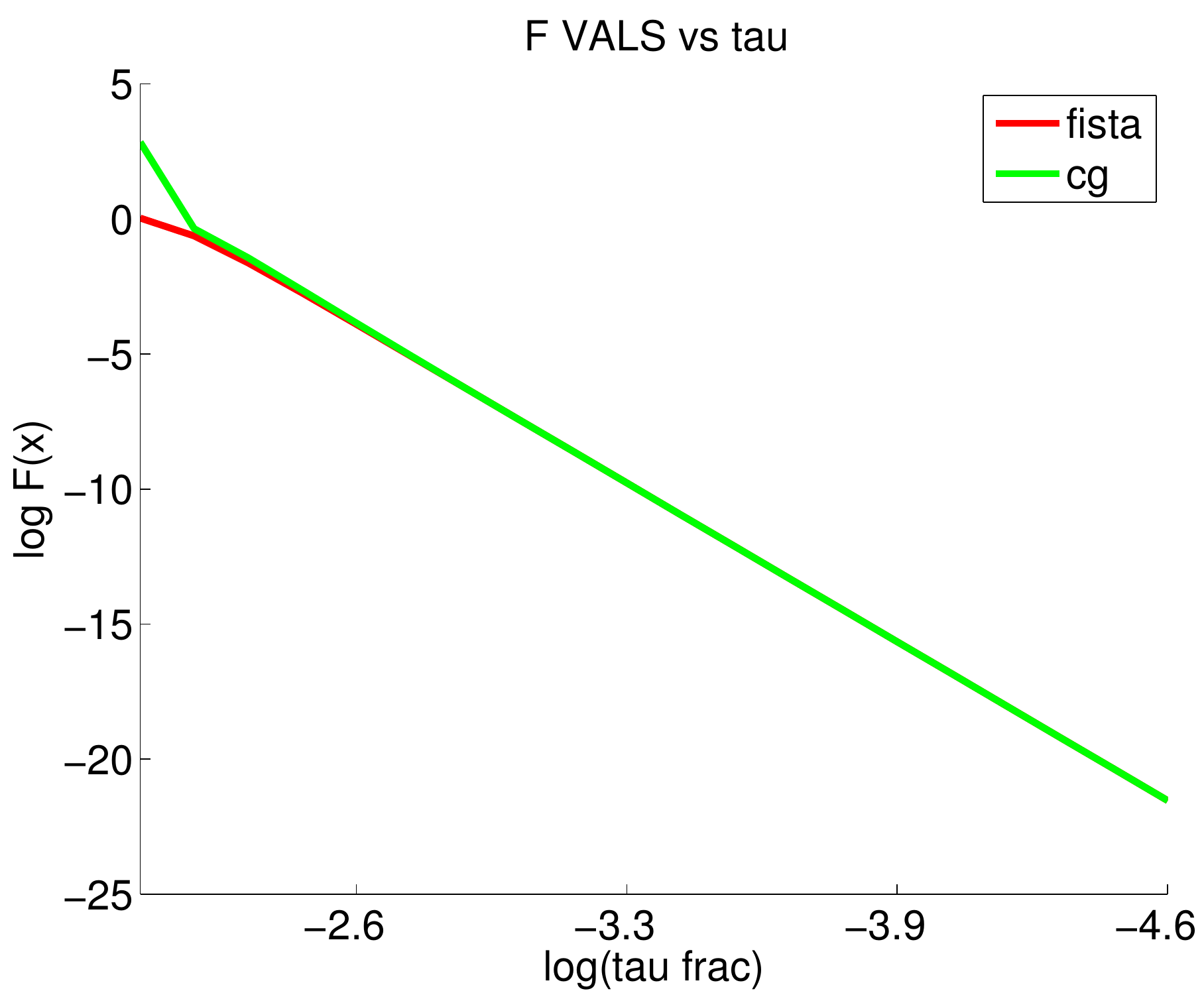}
}
\centerline{
\includegraphics[scale=0.25]{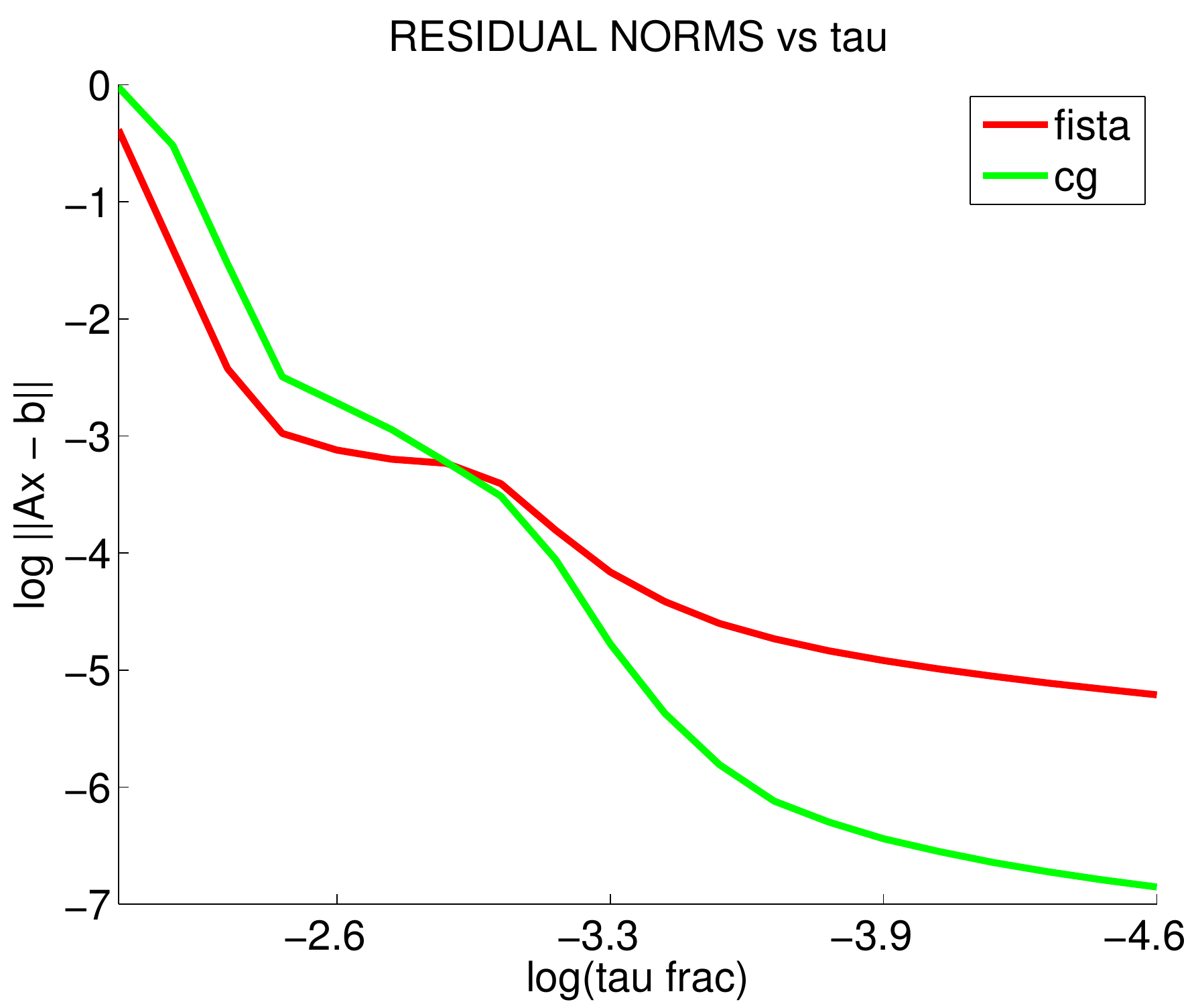}
\includegraphics[scale=0.25]{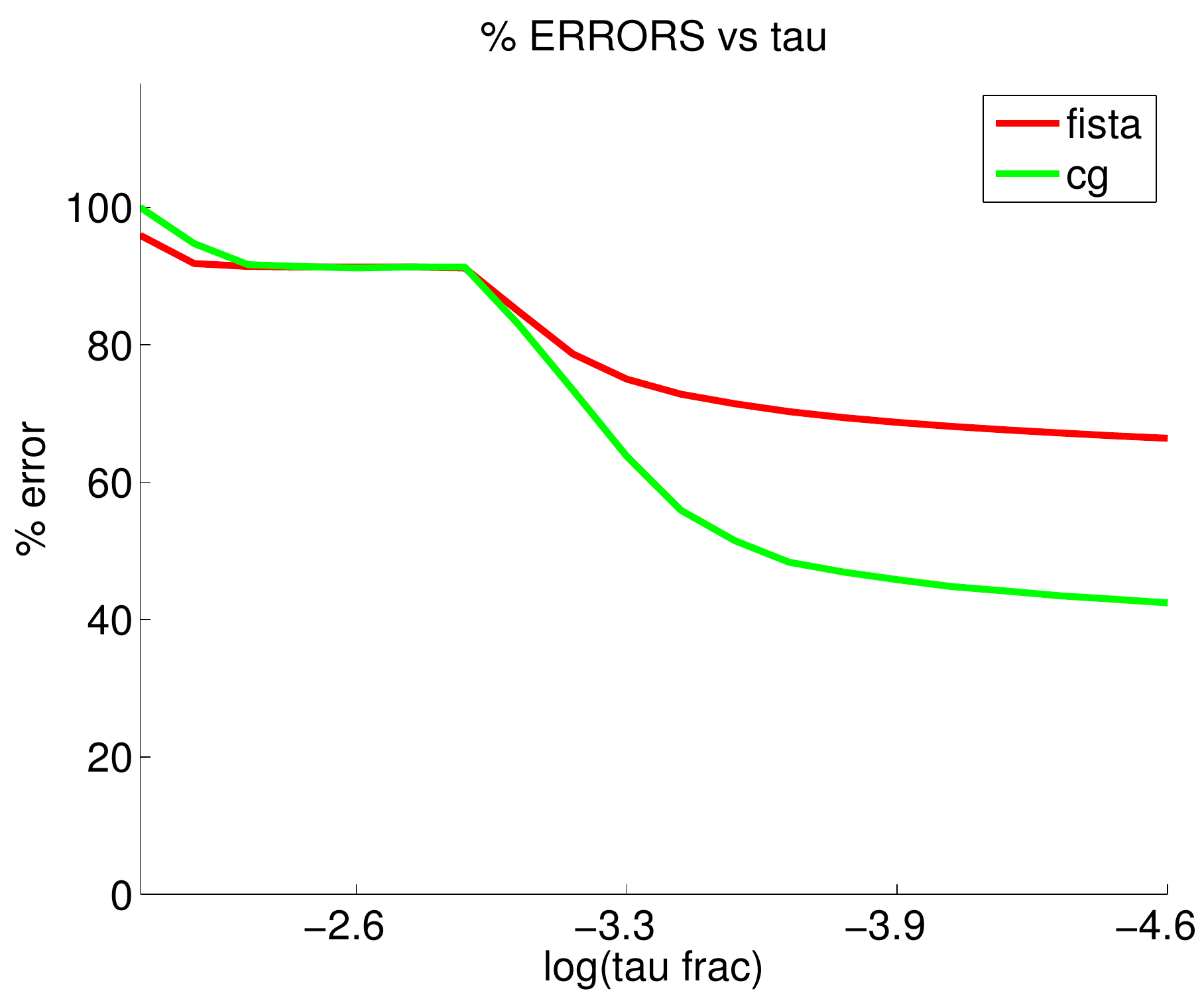}
\includegraphics[scale=0.25]{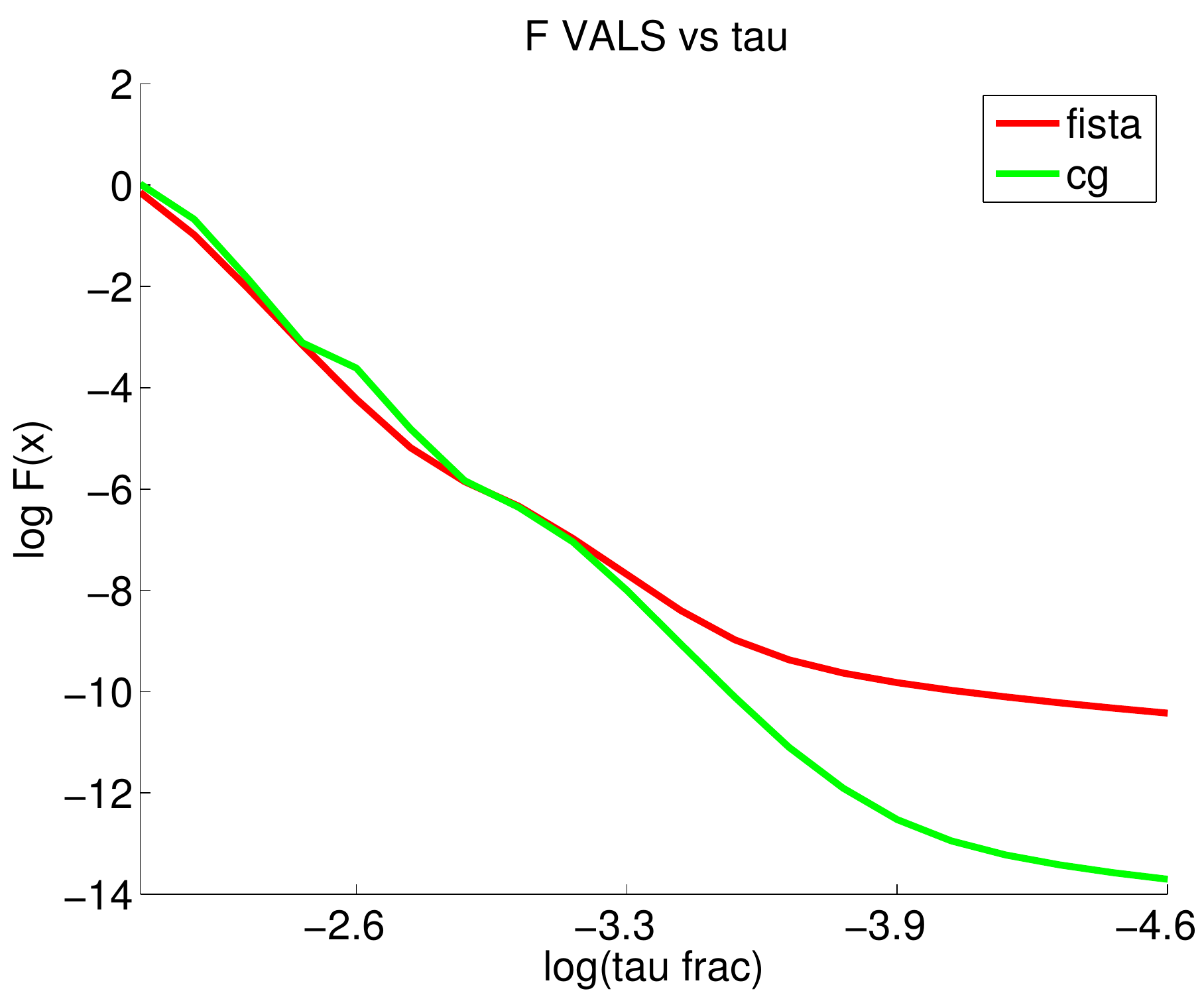}
}
\end{figure*}

\newpage

\begin{figure*}[ht!]
\centerline{
\includegraphics[scale=0.25]{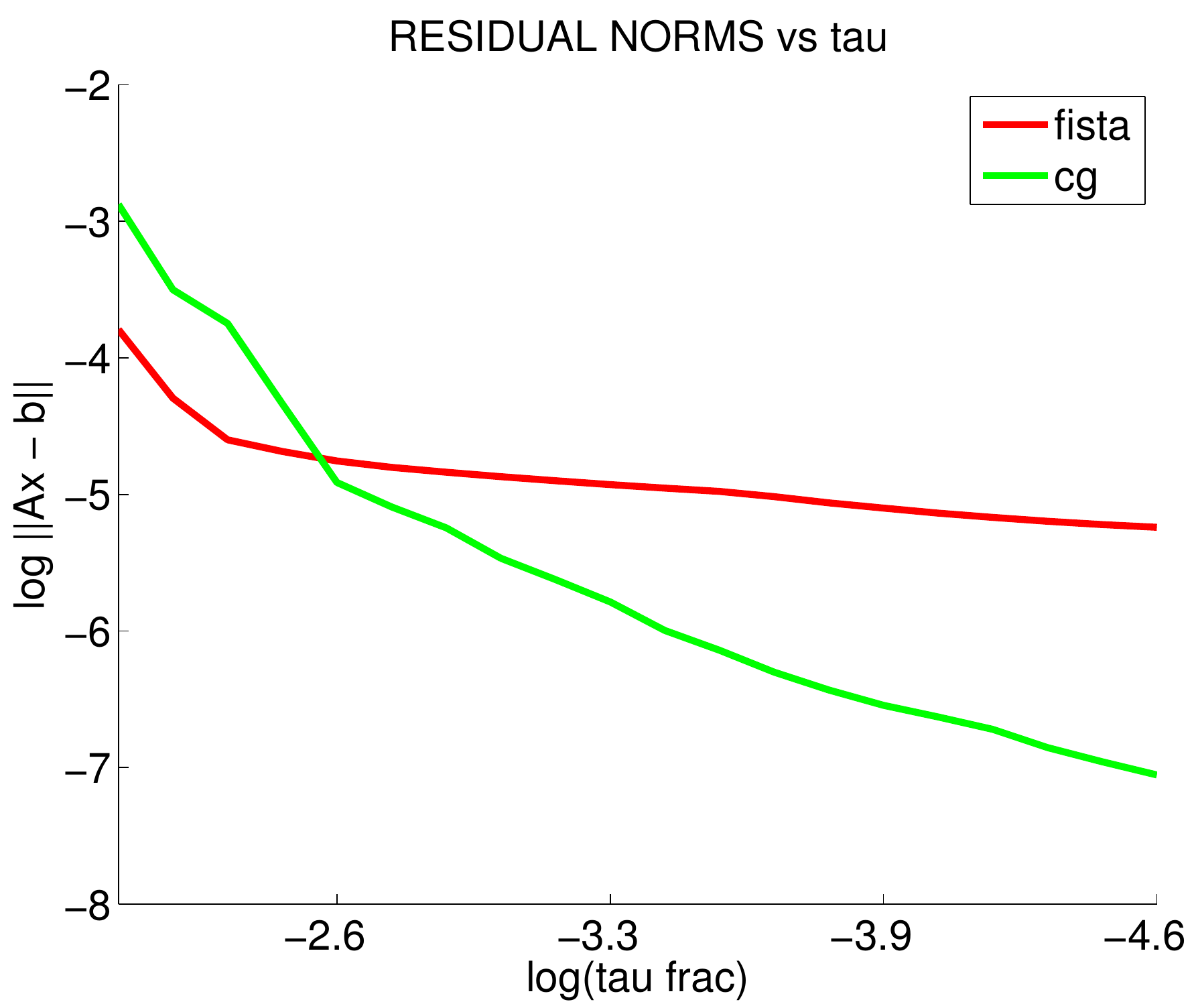}
\includegraphics[scale=0.25]{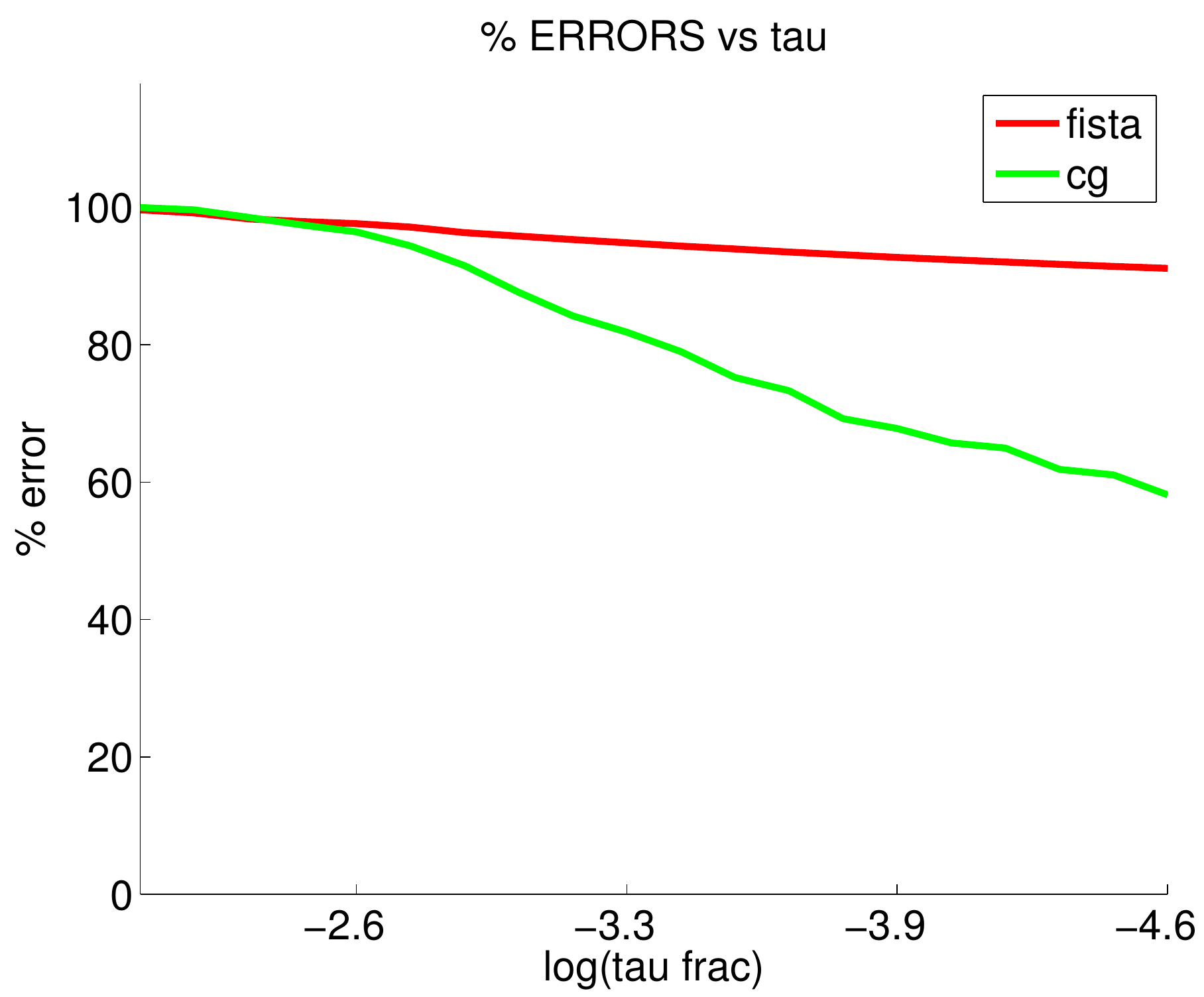}
\includegraphics[scale=0.25]{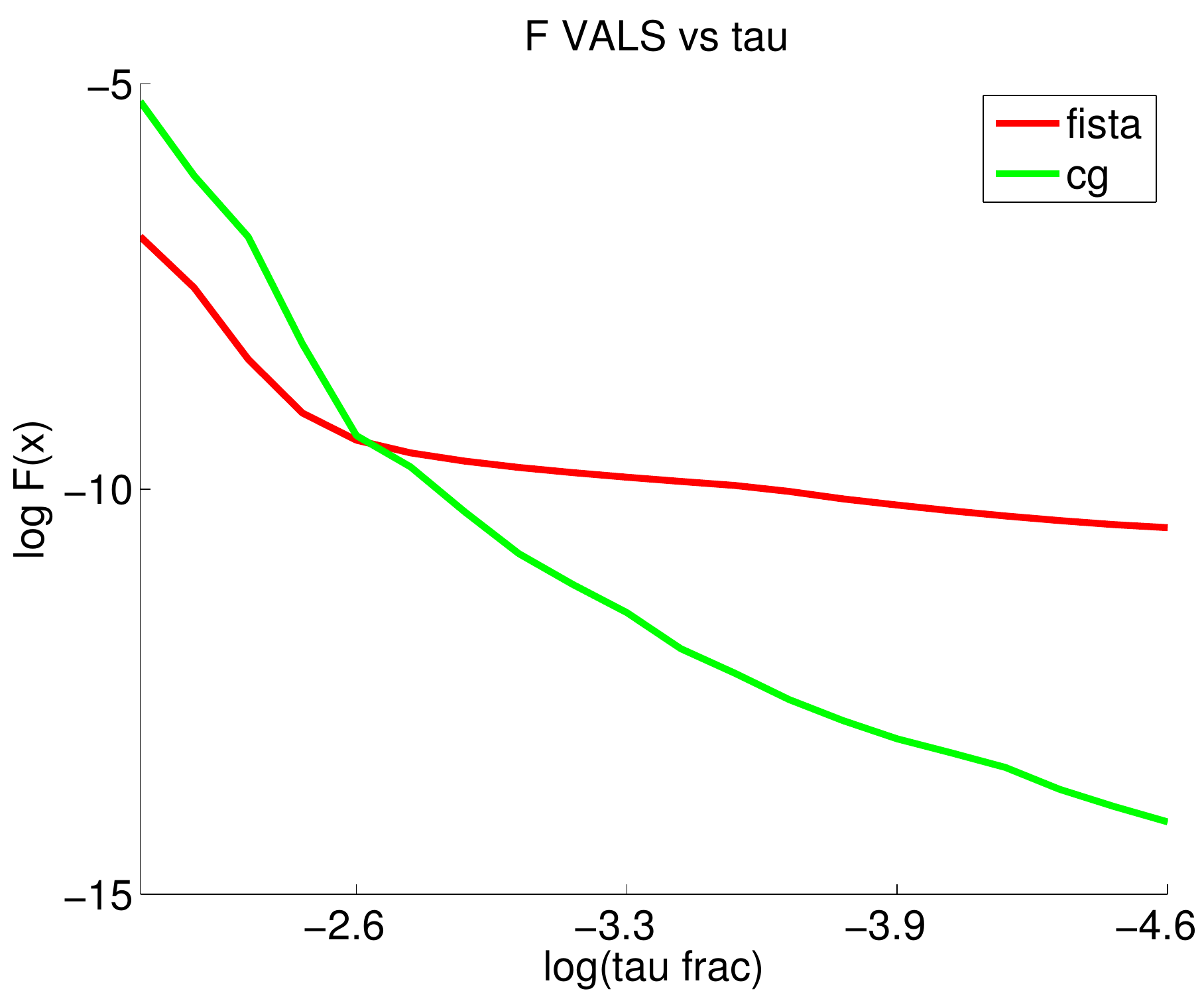}
}
\caption{Averaged quantities along the regularization curve for matrix types I,II,III. 
Residual norms $||Ax_{\tau} - b||_2$, percent 
errors $100 \frac{||x_{\tau} - x||_2}{||x||_2}$, and $\ell_1$ functional values $F_1(x_{\tau})$ 
versus decreasing $\tau$ (fraction of $||A^T b||_{\infty}$) for FISTA and CG.}
\label{fig:sparse_reconstruction2}
\end{figure*}

\newpage

\begin{figure*}[!ht]
\centerline{
\includegraphics[scale=0.25]{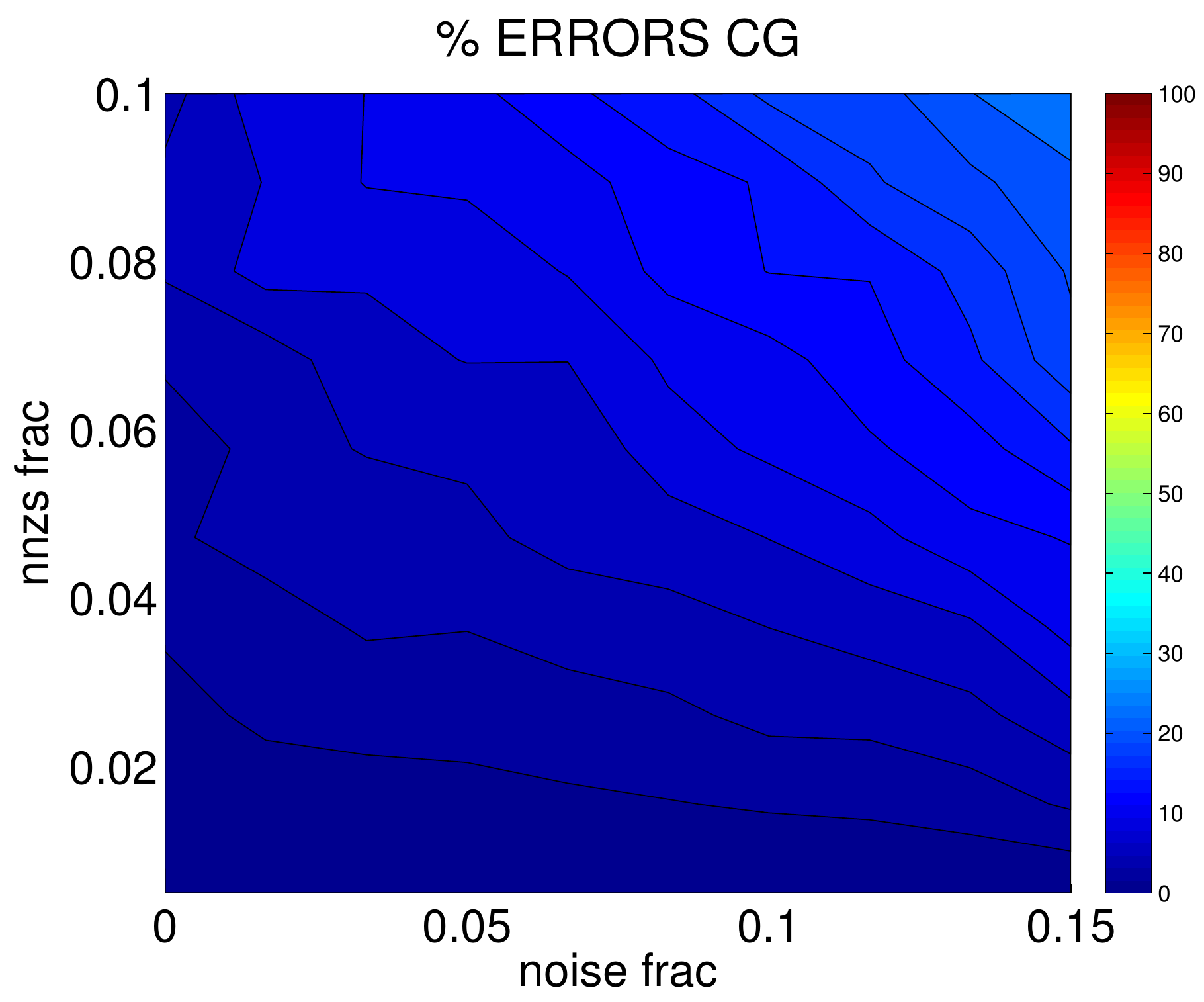}
\includegraphics[scale=0.25]{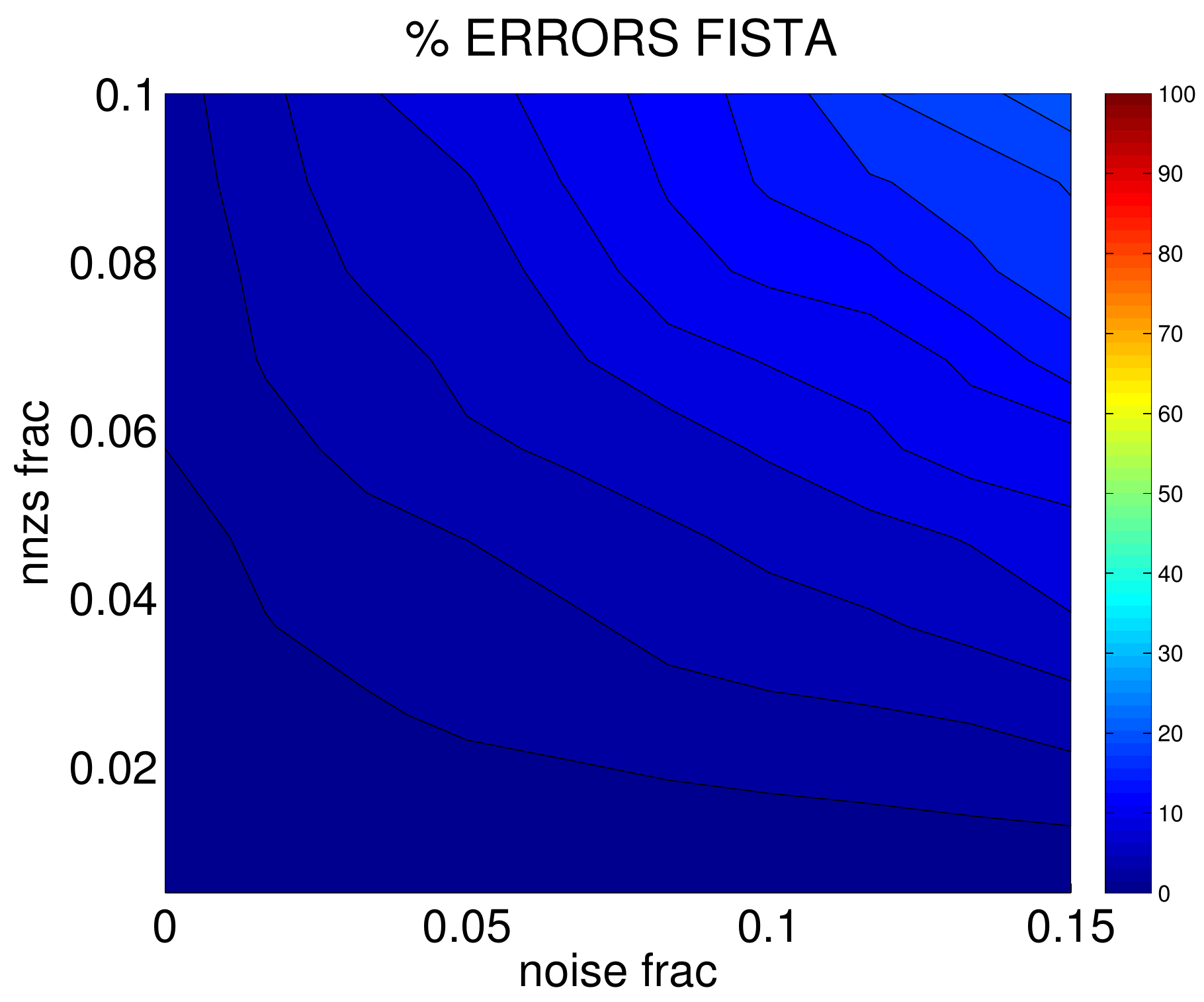}
}
\centerline{
\includegraphics[scale=0.25]{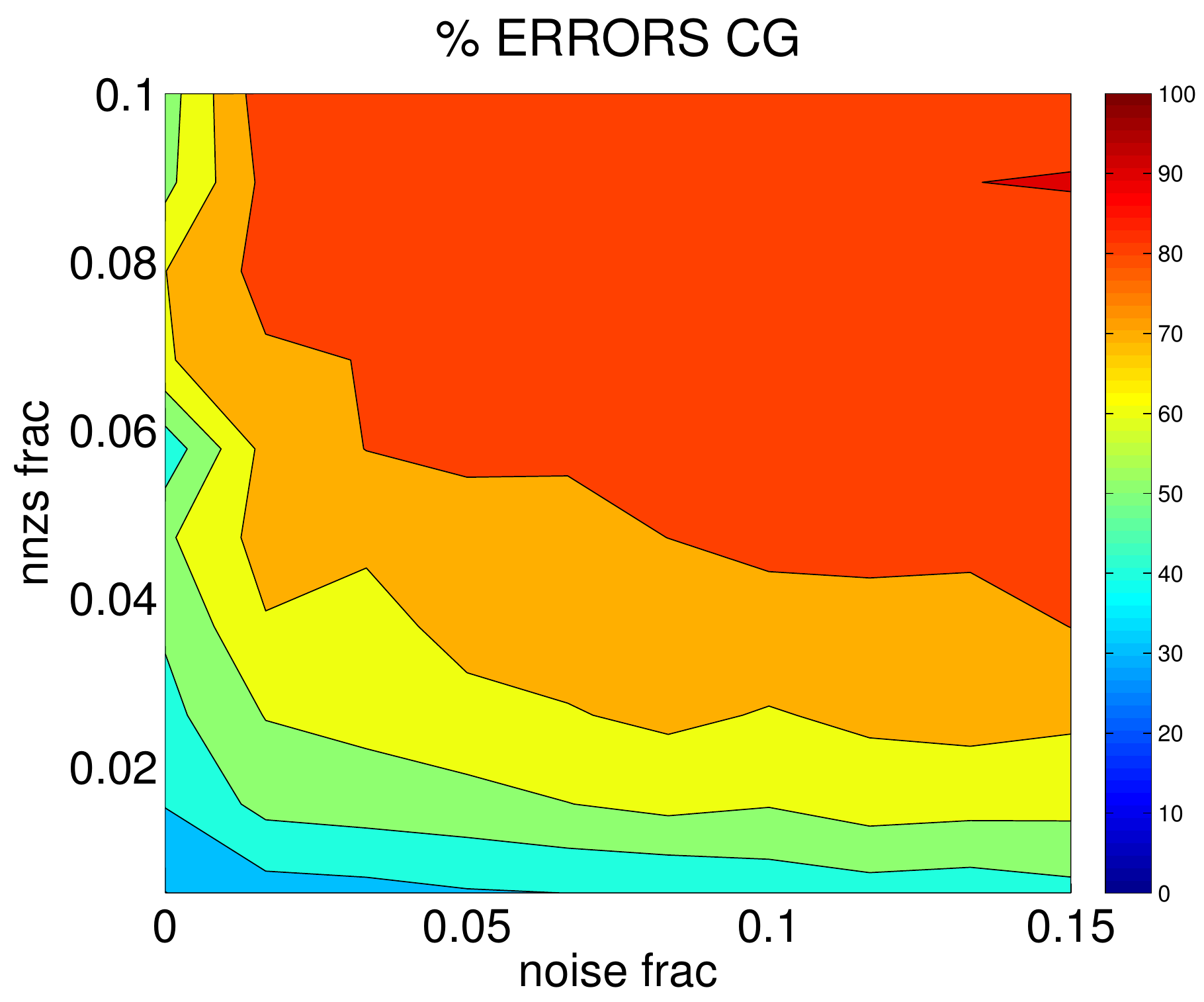}
\includegraphics[scale=0.25]{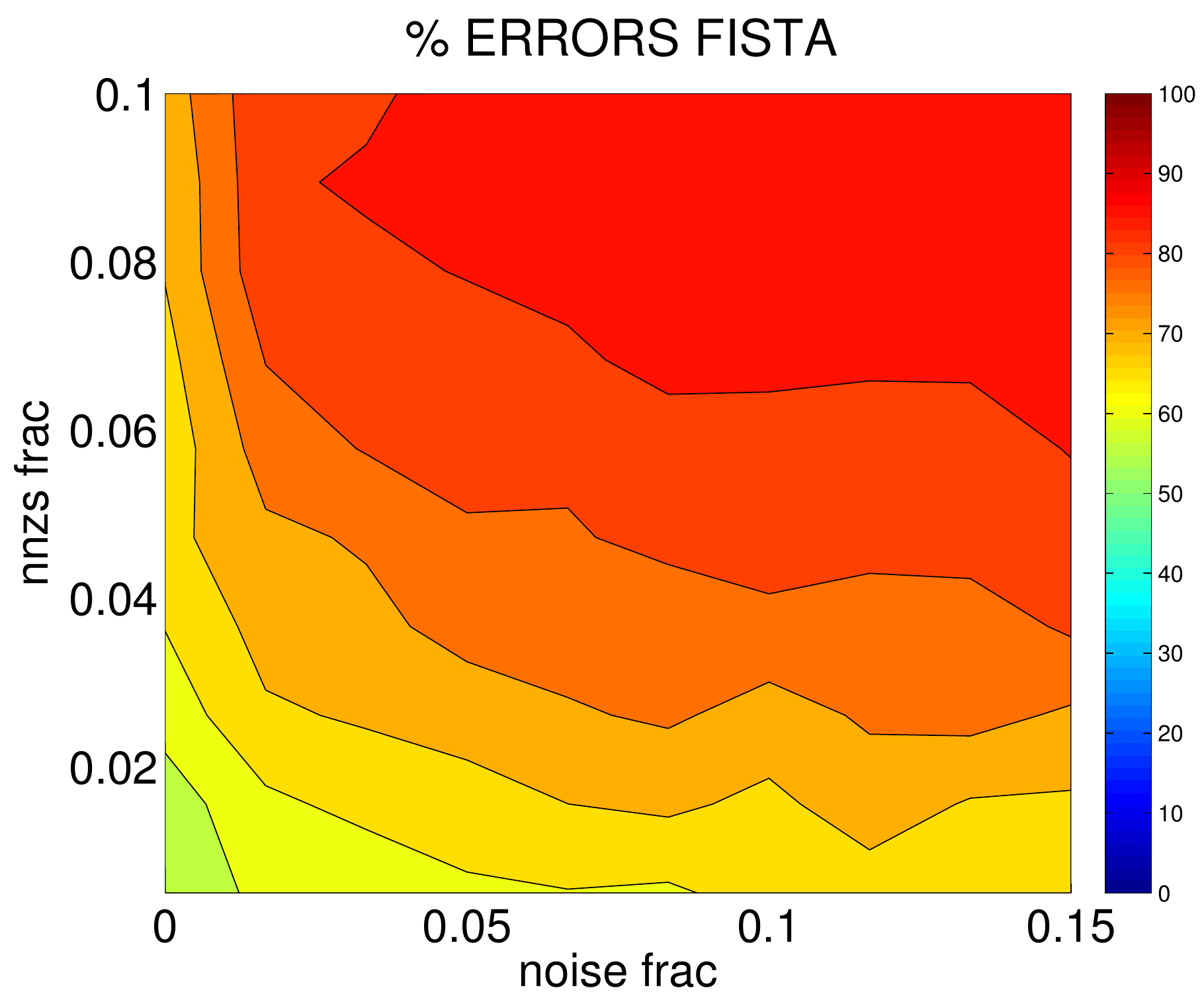}
}
\centerline{
\includegraphics[scale=0.25]{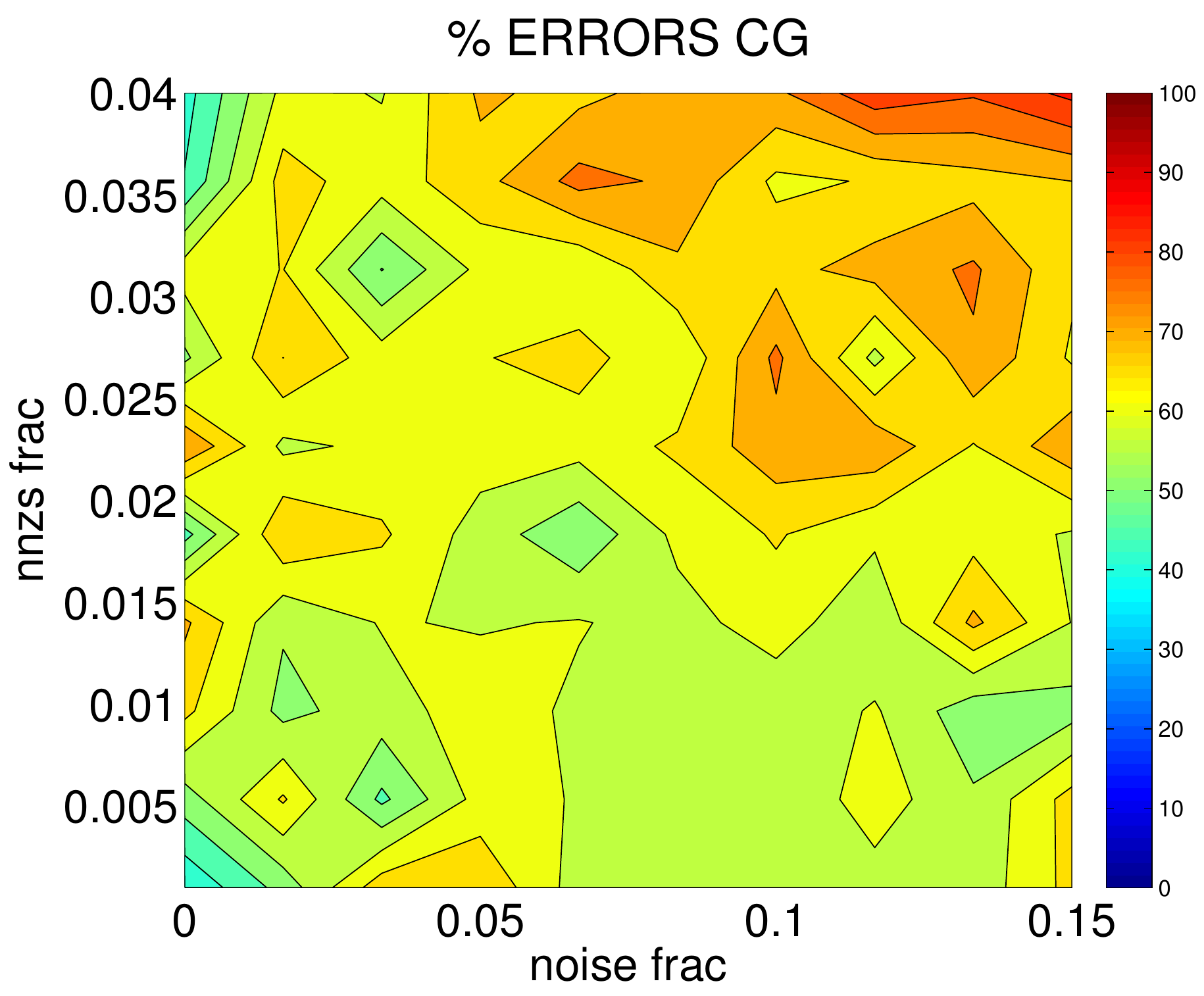}
\includegraphics[scale=0.25]{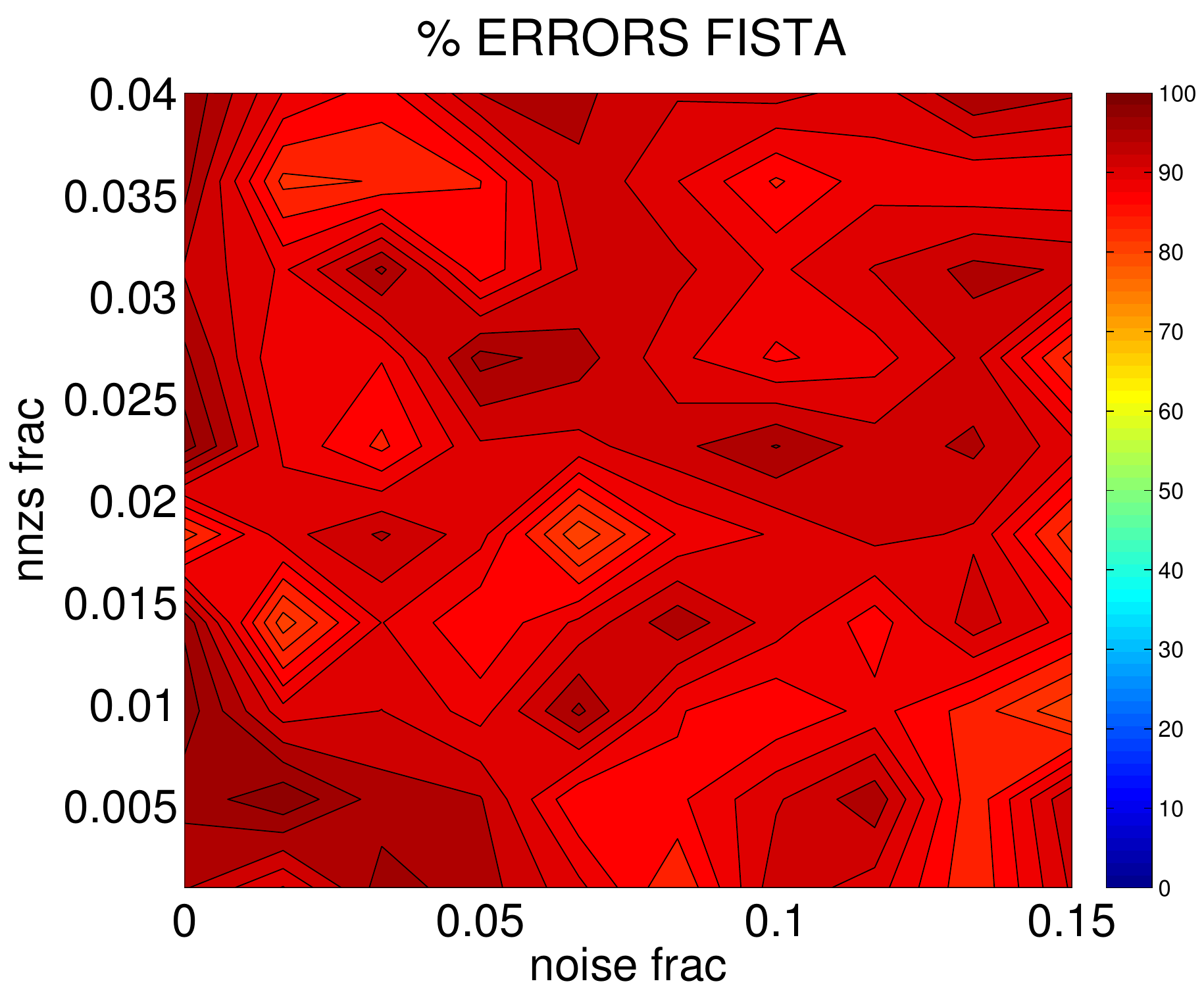}
}
\caption{Contour plots for minimum percent errors over all $\tau$'s along the regularization curve 
at different combinations of nonzeros and noise fractions for matrix types I,II,III.}
\label{fig:sparse_reconstruction3}
\end{figure*}

\newpage

\begin{figure*}[!ht]
\centerline{
\includegraphics[scale=0.10]{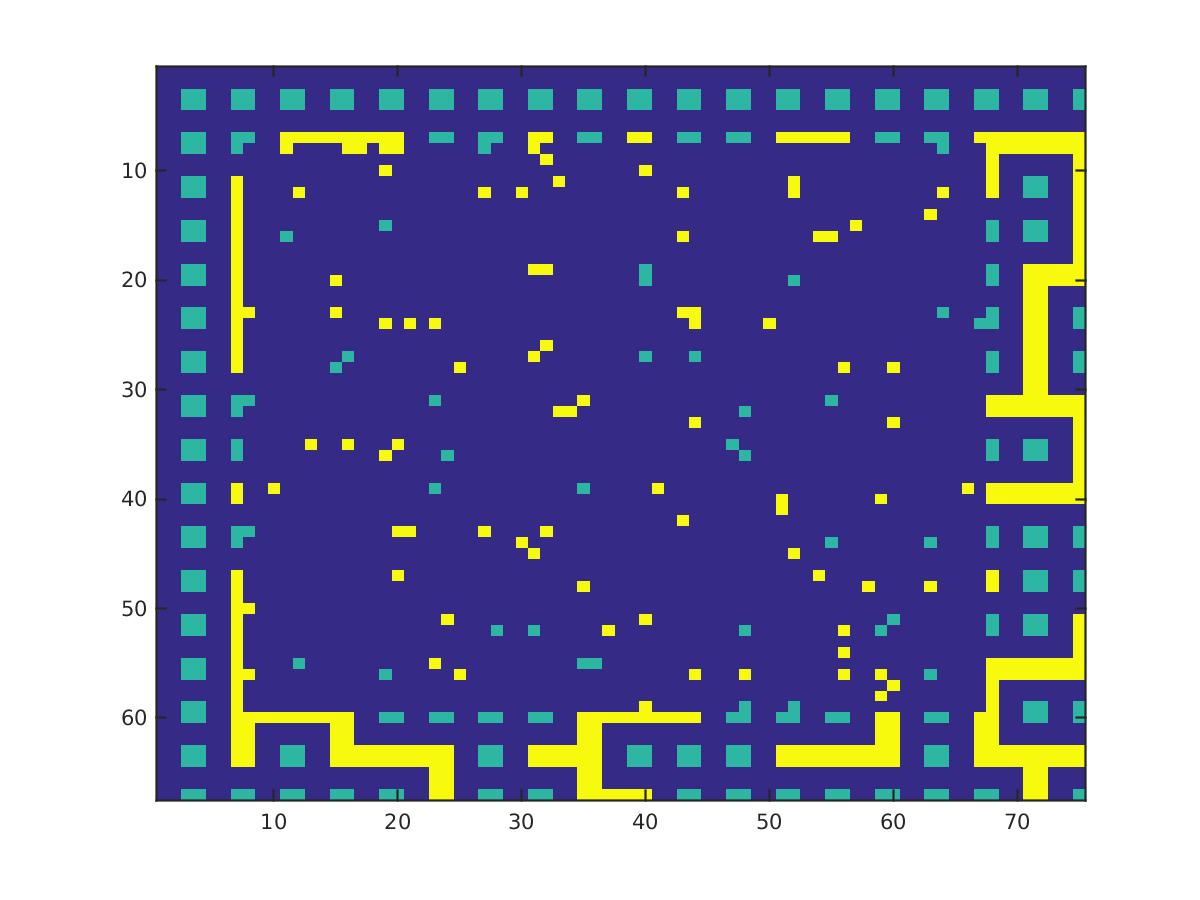}
\includegraphics[scale=0.10]{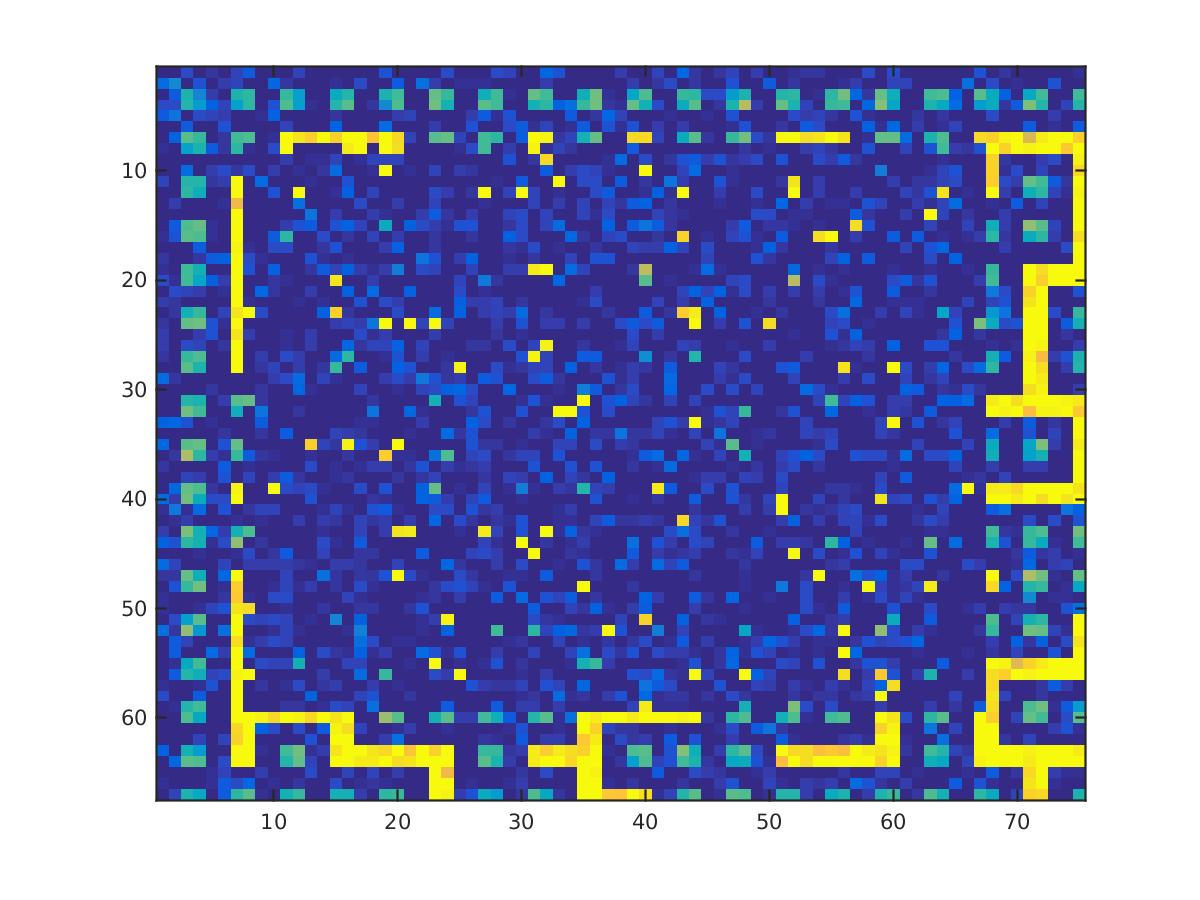}
}
\vspace{3.mm}
\centerline{
\includegraphics[scale=0.10]{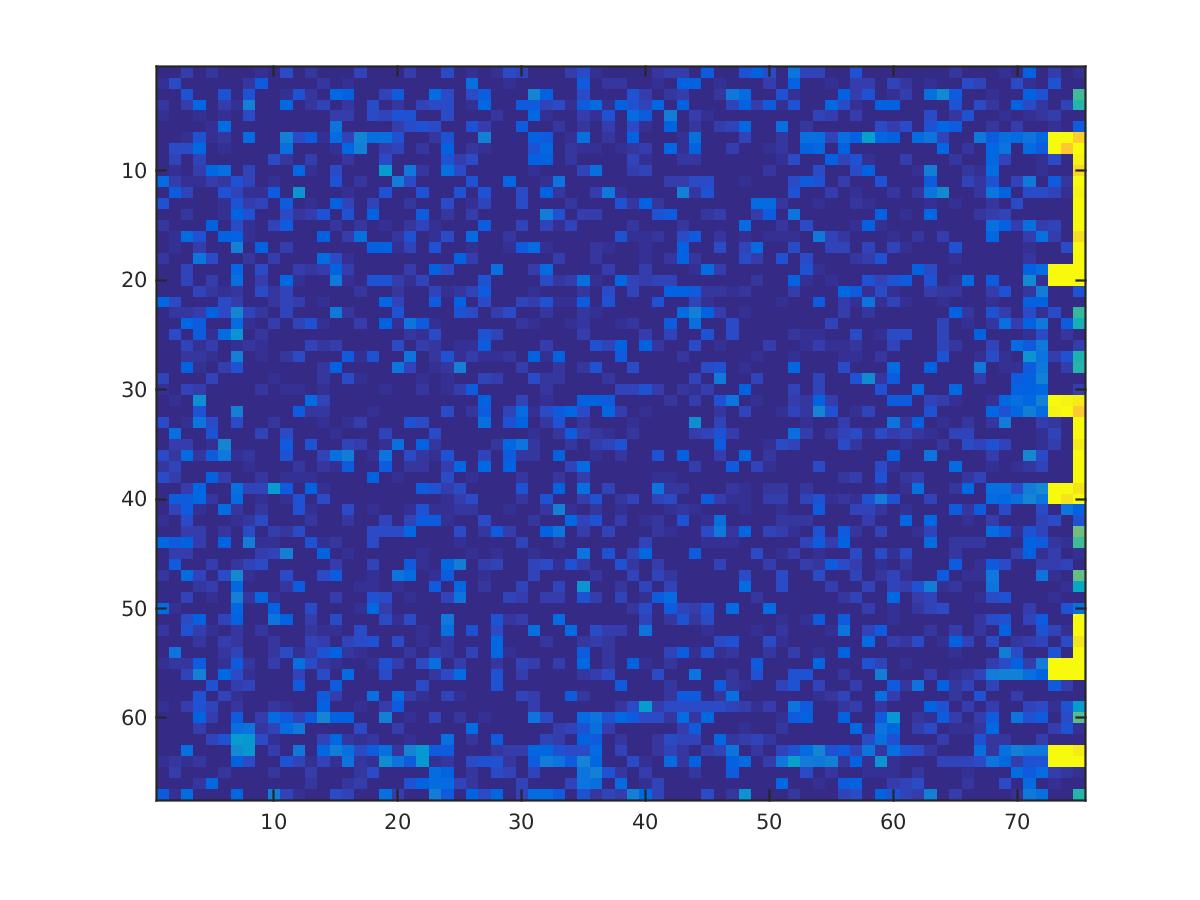}
\includegraphics[scale=0.10]{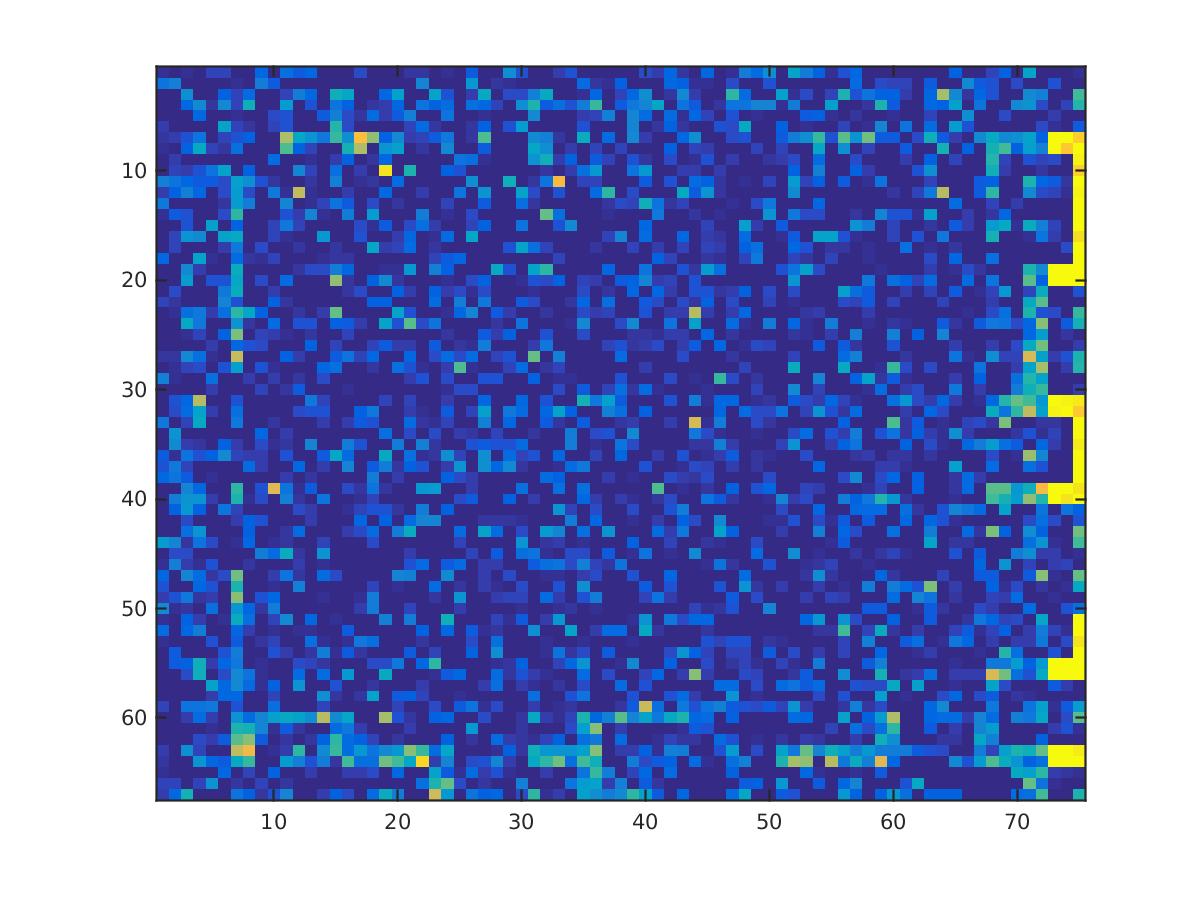}
\includegraphics[scale=0.10]{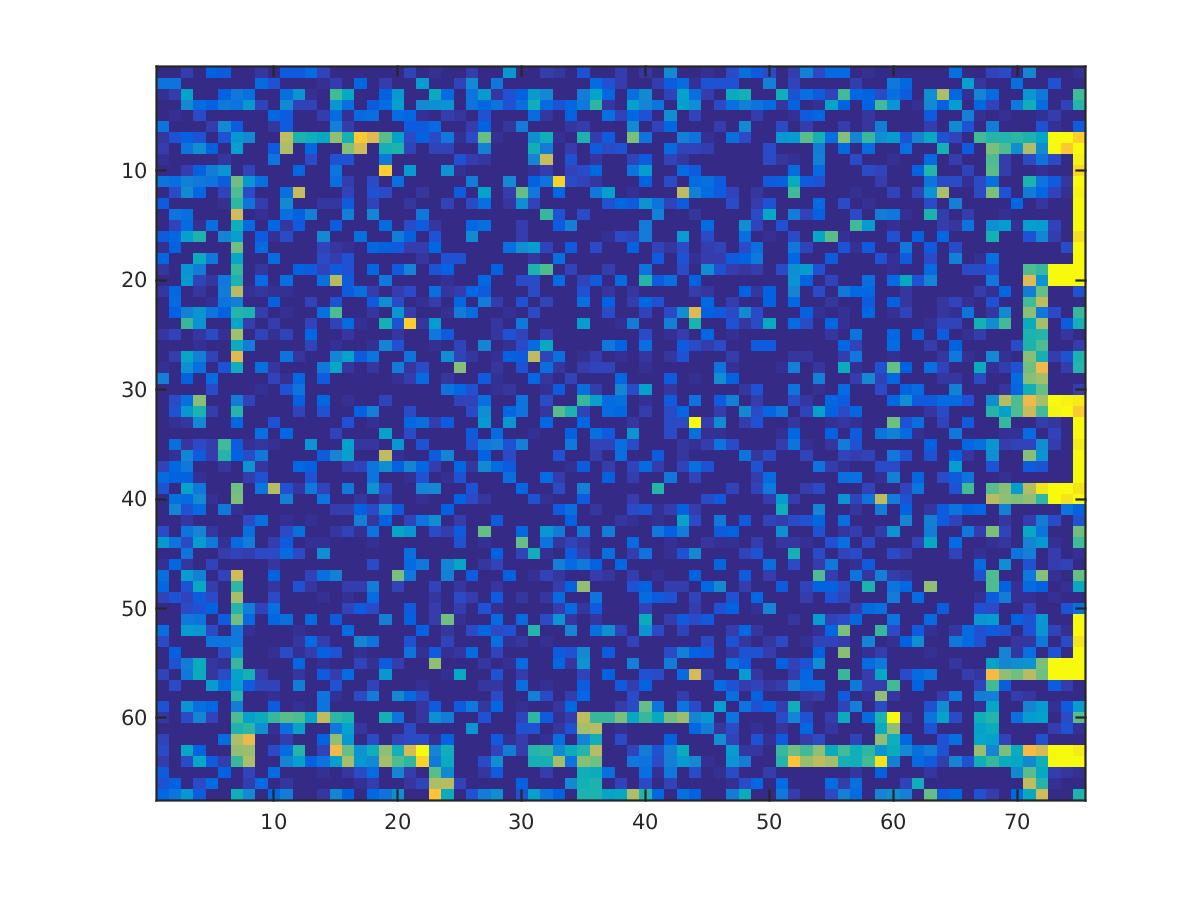}
\includegraphics[scale=0.10]{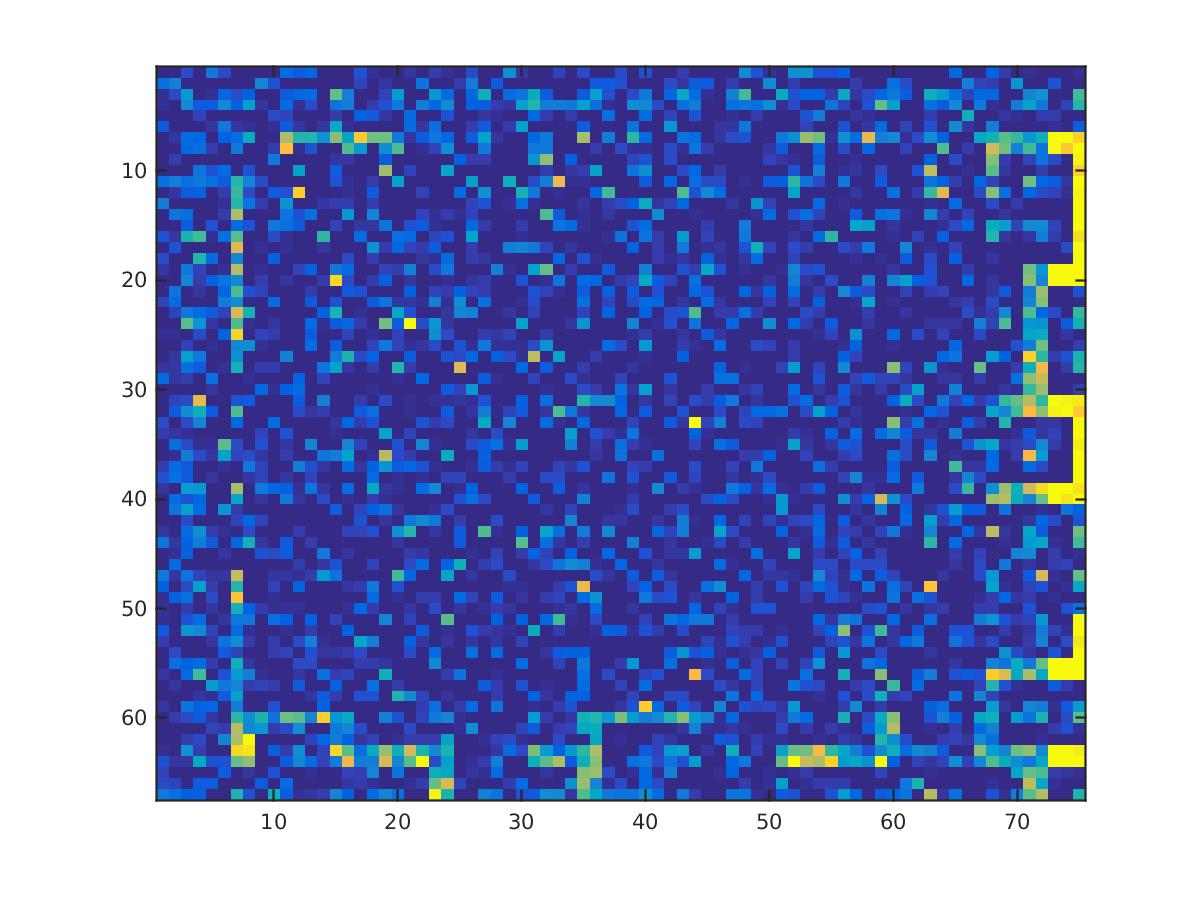}
}
\vspace{5.mm}
\centerline{
\includegraphics[scale=0.18]{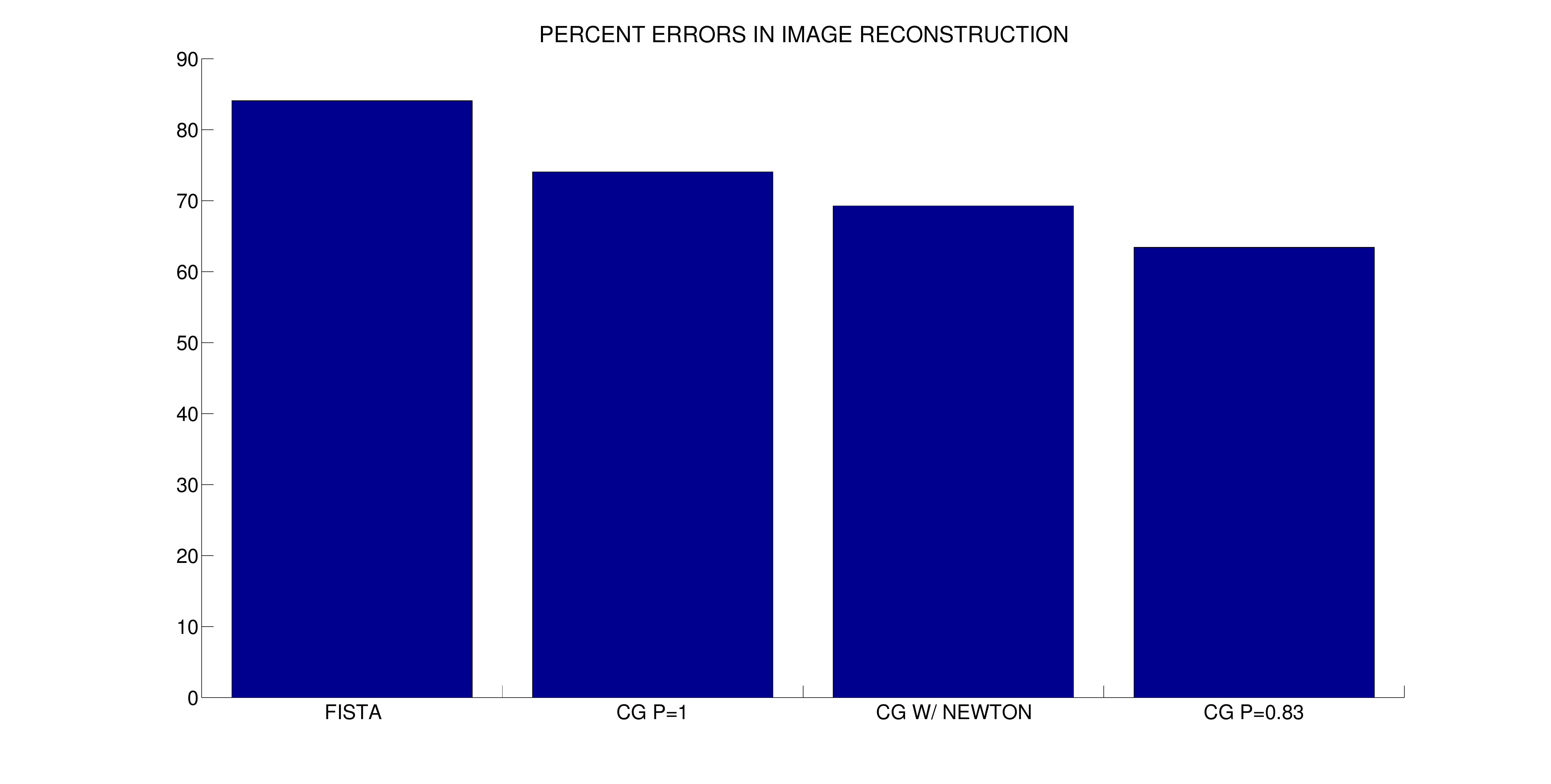}
}
\caption{Image reconstruction for a $5025$ pixel image sampled 
via a random sensing matrix of size 
$5000 \times 5025$ with rapidly decaying singular values, applied to the noisy version of the 
image. Row 1: original and noisy image. Row 2: recovered images using 
FISTA, CG (with $p=1$), CG with Newton's method (with $p=1$), and 
CG (with $p=0.83$) run for $100$ iterations at each $\tau$ (FISTA) and at $50$ 
iterations (CG). Row 3: bar plot of percent errors 
(recovery error between recovered and original image) for the different algorithms.} 
\label{fig:sparse_reconstruction4}
\end{figure*}

\newpage
\section{Conclusions}
In this article, we proposed new convolution based smooth approximations for the 
absolute value function $g(t) = |t|$ using the concept of  
approximate mollifiers. We established convergence results for our approximation 
in the $L^1$ norm. We applied the approximation to the minimization of the non-smooth 
functional $F_p(x)$ which arises in sparsity promoting regularization 
(of which the popular $\ell_1$ functional is a special case for $p=1$) to 
construct a smooth approximation $H_{p,\sigma}(x)$ of $F_p(x)$ and derived the 
gradient and Hessian of $H_{p,\sigma}(x)$. 

We discussed the use of the nonlinear CG algorithm and higher order algorithms (like Newton's 
method) which operate with the smooth 
$H_{p,\sigma}(x)$, $\nabla H_{p,\sigma}(x)$, $\nabla^2 H_{p,\sigma}(x)$ functions 
instead of the original non-smooth functional $F_p(x)$.

We observe from the numerics in Section \ref{sect:numerics} that in many cases, 
in a small number of iterations, we are able to obtain better results 
than FISTA can for the $\ell_1$ case (for example, in the presence of 
high noise). We also observe that when $p<1$ but not too far away from one 
(say $p \approx 0.9$) we can sometimes obtain even better reconstructions 
in compressed sensing experiments. 

The simple algorithms we show maybe useful for larger problems, where 
one can afford only a small number of iterations, or when one wants to quickly 
obtain an approximate solution (for example, to warm start a thresholding based method).  
The presented ideas and algorithms can be applied to design more complex 
algorithms, possibly with better performance for ill-conditioned problems, 
by exploiting the wealth of available literature on conjugate gradient and 
other gradient based methods. Finally, the convolution smoothing technique which 
we use is more flexible than the traditional mollifier approach and 
maybe useful in a variety of applications where the minimization of 
non-smooth functions is needed.

\newpage

\bibliographystyle{plain}
\bibliography{ref}

\end{document}